\documentclass[11pt]{article}
\usepackage{hyperref}
\usepackage{amsmath, graphicx, latexsym, amssymb, amsthm, amsfonts}
\usepackage[mathscr]{eucal}
\usepackage{graphics}
\usepackage{color}
\usepackage{epsfig}
\usepackage{enumerate}
\usepackage{mathtools}

\mathtoolsset{showonlyrefs}

\newtheorem{lemma}{Lemma}[section]
\newtheorem{theorem}[lemma]{Theorem}
\newtheorem{corollary}[lemma]{Corollary}

\newtheorem{proposition}[lemma]{Proposition}

\newtheorem{remark}[lemma]{Remark}

\allowdisplaybreaks

\newcommand{\xbf}{\textbf{x}}

\newcommand{\eps}{{\varepsilon}}

\newcommand{\Emb}{{\mathbb{E}}}

\newcommand{\NN}{{\mathbb{N}}}

\newcommand{\RR}{{\mathbb{R}}}
\newcommand{\EE}{{\mathbb{E}}}
\newcommand{\PP}{{\mathbb{P}}}

\newcommand{\clb}{{\mathcal{B}}}

\newcommand{\clp}{{\mathcal{P}}}

\newcommand{\cld}{{\mathcal{D}}}

\newcommand{\clf}{{\mathcal{F}}}

\newcommand{\cls}{{\mathcal{S}}}

\newcommand{\Pbf}{{\mathbf{P}}}

\newcommand{\Xbf}{{\mathbf{X}}}
\newcommand{\Ybf}{{\mathbf{Y}}}
\newcommand{\Zbf}{{\mathbf{Z}}}

\newcommand{\one}{{\boldsymbol{1}}}

\newcommand{\Om}{\Omega}

\newcommand{\bfX}{\mathbf{X}}
\newcommand{\bfY}{\mathbf{Y}}
\newcommand{\bfZ}{\mathbf{Z}}

\newcommand{\bfz}{\mathbf{z}}
\newcommand{\bfx}{\mathbf{x}}
\newcommand{\bfy}{\mathbf{y}}

\newcommand{\cll}{\mathcal{L}}
\newcommand{\clg}{\mathcal{G}}

\newcommand{\be}{\begin{equation}}
\newcommand{\ee}{\end{equation}}
\newcommand{\beq}{\begin{align*}}
\newcommand{\eeq}{\end{align*}}

\newcommand{\ybf}{\textbf{y}}
\newcommand{\zbf}{\textbf{z}}
\newcommand{\veps}{\varepsilon}
\newcommand{\dbl}{d_{\mbox{\tiny{BL}}}}
\numberwithin{equation}{section}
\begin{document}

\title{Long Time Asymptotics for the Stochastic Follow-the-Leader System}
\author{Sayan Banerjee, Amarjit Budhiraja and Dilshad Imon}
\date{}
\maketitle
\begin{abstract}
We introduce and analyze a class of interacting particle systems on the real line that
combine features of the stochastic rat race and (deterministic) follow-the-leader models. The particle system
evolves as a continuous-time pure jump process: the leading particle moves independently, at Exponential jump times,
with constant jump rate and iid jump sizes distributed
according to a law~$\theta$, while each of the remaining particles jumps forward, at Exponential times, at 
rate equal to its distance from the particle immediately ahead, with jump sizes drawn
uniformly from the corresponding gap. The dynamics thus encode competition for
leadership together with distance-dependent stochastic interactions. Our main focus is the associated \emph{gap process}, representing the vector of
inter-particle distances. We establish the existence of a unique stationary distribution for
the gap process and
 prove {\em uniform} geometric ergodicity. Further, when the
leader's jump sizes follow an Exponential distribution, we identify the stationary law
explicitly as a product of independent Exponential laws. For this case, we
also derive bounds on the mixing time, showing that it scales between
$\Theta(n)$ and $O(n(\log n)^2)$ for an $n$-particle system. As an application of the mixing time results we establish a functional limit theorem that characterizes fluctuations of particle states at large time, under a suitable spatial and temporal scaling and large particle limit. Finally, when the leader's
jumps have heavy but integrable tails, we show that each gap has at least one additional finite moment under stationarity than that of the leader's jump size distribution. Together, these results provide a comprehensive analysis of stability
and convergence to equilibrium in an interacting jump system with state-dependent
rates and jump sizes, and local interactions. The model offers a tractable setting for exploring ergodicity, explicit invariant
laws, and mixing behavior in non-diffusive particle systems.
\noindent\newline


\end{abstract}

\section{Introduction}

\subsection{Motivation and model}
The dynamics of complex systems often involve local interactions between individual agents, leading to emergent collective behavior. In this article, we focus on a specific such type of interaction between $n$ particles on the real line where each particle `follows' the particle (or particles) ahead of it, leading to `flocking' type behavior in the full system.
Two distinct existing modeling approaches used to understand such phenomena are the \emph{stochastic rat race} and \emph{follow-the-leader} models. 

The stochastic rat race model is a  framework exploring competitive dynamics within individuals in a population. For instance, \cite{ben2007toy} introduces a system where any particle jumps forward by a random length uniformly chosen from 0 to its distance from the leader. The leader seeks to maintain its position by jumping by a length randomly chosen from a uniform distribution with width given by its lead with respect to the next strongest individual. All particles have the same constant jump rate.

On the other hand, follow-the-leader models are deterministic ODE based systems that describe scenarios where individuals adjust their behavior based on those of a designated leader or a group of leading members.  In traffic flow, follow-the-leader models analyze how vehicles adjust their speed and spacing based on the vehicle ahead, and can be viewed as a numerical method to compute solutions for traffic flow (\cite{gazis1961nonlinear}, \cite{ridder2018traveling}, \cite{holden2017follow}). The paper \cite{gazis1961nonlinear} introduces such a model where the speed of a vehicle is a non-linear function of the distance from the one immediately ahead of it - that is, each driver adjusts their speed based on that of the speed of the car in front. This model also includes an overall leader (car farthest in front) who moves at a speed that is not influenced by the state of the remaining system. Extensions to this classical model include \cite{ridder2018traveling} which studies a traffic flow model where the velocity of a given car depends on the relative positions of all other cars which are present within a given distance in front of it. 

Motivated by these two strands of work, we introduce and analyze in this paper an
interacting particle system that combines features of both paradigms while also
introducing new features to the dynamics. Our model evolves as a pure jump process on the real line. The
leading particle, or ``leader", moves forward according to a renewal reward process,
jumping at constant rate with independent jump sizes distributed according to a law
$\theta$. The remaining particles advance in response to their immediate predecessor:
the $i^{th}$ particle jumps forward at a rate equal to the distance between itself and the
particle directly ahead. Further, the new position after a jump is drawn uniformly from the
interval between its current location and that of the particle in front. The system therefore
retains the asymmetry of a designated leader, as in follow-the-leader models, while
introducing competitive, distance-dependent stochastic dynamics similar to the rat
race framework. Such dynamics can be used to model various phenomena such as traffic flow, biological flocking and transport of energy through one-dimensional media. This also provides a natural model for a series of servers processing an infinite fluid workload, where work must pass sequentially through all servers before exiting the system, and where non-leading servers adjust their processing rates in response to their local queue lengths. We refer to this model as the \emph{Stochastic Follow-the-Leader system}.

There is a connection between the Stochastic Follow-the-Leader system and \emph{Hammersley's process}, which was introduced and analyzed in \cite{hammersley1972few,aldous1995hammersley}, and subsequently studied by \cite{krug2000asymmetric,seppalainen2001perturbation,cator2005ham}, among many others, motivated by asymptotic properties of longest increasing subsequences in uniform random permutations of $\{1,\ldots,n\}$ as $n \to \infty$. Several variants of this process have been studied in the literature, and we give a description that is closest to our system. Start with a point process of particles $\mathcal{P}(0)$ on the real line and construct the process $\{\mathcal{P}(t) = (P_j(t) : j \in \mathbb{Z}) \,:\, t \ge 0\}$, with $P_j(t) >P_{j+1}(t)$ for all $j \in \mathbb{Z},\, t \ge 0$, started from $\mathcal{P}(0)$ and updated by sliding a horizontal line upwards through a space-time Poisson point process of unit intensity on $\mathbb{R}^2$ as follows. For $t >0$, if a point $P$ appears on the horizontal line $y=t$ between two successive points $P_i(t-)$ and $P_{i+1}(t-)$ in $\mathcal{P}(t-)$, then $P_{i+1}(t) = P$ and $P_j(t) = P_j(t-)$ for all $j \neq i+1$. This process can be viewed as the Stochastic Follow-the-Leader system with an \emph{infinite} number of particles, indexed by $\mathbb{Z}$ (without a leader), as any particle $P_{j+1}$ jumps forward at rate $P_j - P_{j+1}$ to a uniformly chosen location in $(P_{j+1},P_j)$. By taking a \emph{local weak limit}, as $n\to \infty$, of the particle system viewed around a uniformly chosen particle in our $n$-particle Stochastic Follow-the-Leader system, one obtains the Hammersley's process. Thus the techniques developed in this paper may be useful in analyzing  properties of the Hammersley's process, such as local stability, rates of convergence and hydrodynamic behavior. See additional comments on this in Section \ref{sec:contr} (item 4).
   
\subsection{Questions and challenges}
The central object of study in this work is the long-time behavior of the associated \emph{gap process}, that is, the
collection of inter-particle distances. The evolution of the gap process is Markovian and encodes all information about the
relative configuration of the particles. Several natural and interesting questions arise: Does the gap
process admit a unique stationary distribution, and if so, under what conditions? How rapidly
does the process converge to equilibrium? What is the structure of the invariant law in
special cases, such as when the leader’s jump sizes are Exponential? And how does the
tail behavior of the leader’s increments influence the regularity of the stationary gaps? 

The study of the gap process presents several challenges. Since the jump rate and jump size of each particle behind the leader is proportional to the gap preceding it, this results in a quadratic non-linearity in the dynamics which is highly sensitive to `very large' and `very small' gaps. When the system has very large gaps, the associated particles exhibit atypically high speeds and large jumps. Small gaps result in significant slow-downs, creating `bottlenecks' in the system. Moreover, as the state space is unbounded and uncountable, with the above described non-linearity and singularity in the dynamics, traditional techniques used in mixing-time analysis of Markov chains (e.g. \cite{levin2017markov}) do not readily apply.   

\subsection{Our contributions}
\label{sec:contr}
We introduce novel techniques to address the above challenges and make the following contributions. 
\begin{enumerate}
  \item In Theorem \ref{thm:stat.ex}, we establish the existence and uniqueness of
a stationary distribution for the gap process for every finite system size $n$, using a Lyapunov function argument combined with the weak Feller property.

  \item In Theorem \ref{thm:unif.erg}, we prove uniform geometric ergodicity, showing that convergence to equilibrium occurs at an exponential rate, and the time to reach within a prescribed distance to stationarity is independent of the initial configuration. This is achieved by establishing a minorization condition and combining it with bounds on exponential moments of hitting times of certain ``small" sets. 
  This uniform ergodicity is a manifestation of the quadratic non-linearity in the system: even for initial configurations `far away' from the stationary gap distribution, the particles quickly organize themselves in order to bring their gap configuration close to equilibrium.
  
  \item In the special case where the leader’s jumps are Exponentially distributed with unit mean, we identify in Theorem \ref{thm:stat.exp.char} the stationary law explicitly as a product of independent rate 1 Exponential laws across all gaps. This key factorization property makes it amenable to a more refined long-time behavior study, discussed in item 4 below. 
  
\item In the Exponential leader case, we investigate \emph{mixing times}, that is, the time required for the distribution of the gap process to become close (in total variation distance) to its stationary law. In Theorem \ref{thm:tmix}, we obtain a $\Theta(n)$ lower bound and $O(n(\log n)^2)$ upper bound for the mixing time, which together characterize it up to logarithmic factors. 

The proof combines coupling constructions with variance
bounds and distinguishing statistics in the spirit of modern Markov chain mixing theory (\cite{levin2017markov}). However, natural couplings do not preserve `monotonicity properties' of the system and we need a novel coupling, described in Section \ref{sec:coupling}. This coupling also plays a key role in the proof of Theorem \ref{thm:unif.erg}.
Moreover, the slow-down caused by small gaps requires careful handling of the worst-case scenario – when all particles are stacked in the same position (all gaps are zero) – and deriving estimates for
the time required for a “reasonably-sized” gap to propagate down the chain of particles that allows for “sufficient” movement of particles. This is achieved through coupling with a related particle system that we call the “frozen boundaries” process (see Section \ref{sec:fbp}).

In Corollary \ref{cor:fclt}, we highlight an application of our mixing time estimates by establishing a functional central limit theorem. We show that, starting from an arbitrary configuration,  an appropriately rescaled and re-centered functional of the gap process at a suitably large time converges in distribution to a standard Brownian motion as $n \to \infty$.

Another possible application of our mixing time estimates is in proving hydrodynamic limits for the {\em stick process} of \cite{seppalainen2001perturbation} which can be regarded as the Hammersley's process as seen from the tagged particle.
The stick process admits a one parameter family (indexed by $\lambda>0$) of stationary distributions given as products of iid Exponential distributions with rate $\lambda$ (see Remark \ref{rem:700n} for some comments on this point).
The paper \cite{seppalainen2001perturbation} proves a hydrodynamic limit given by Burger's equation when the initial configuration is a small perturbation of a stationary profile of the stick process for some $\lambda>0$, given in terms of independent Exponential laws for the gaps (i.e. the `sticks' in the stick process), with a suitable centering and space-time scaling. We expect that the techniques developed in obtaining the mixing time results in the current work will be useful for establishing similar hydrodynamic behavior  for more general initial profiles. This will be investigated in future work.

 \item Finally, we extend our analysis to the setting where the leader’s jumps have heavy, but integrable, tails. In this regime, explicit formulae for the stationary law are no longer
available. Nevertheless, we show in Theorem \ref{thm:poly_tails} that if the leader’s increment distribution has a finite
$k^{th}$ moment, then every gap has a finite $(k+1)^{th}$ moment at stationarity. This result exhibits a form of \emph{regularization} where the randomness of the leader’s motion is `smoothed out' by the motion of the followers.
\end{enumerate}

\subsection{Other related works}
This work studies interacting stochastic particle systems in which interactions depend on particles’ relative positions. Such models—termed \emph{particle systems with topological interaction} by \cite{CDGP}—arise in diverse areas including ecology, evolutionary biology, engineering, and mathematical finance. A canonical example is the \emph{Atlas model}, a special case of rank-based diffusions originating in stochastic portfolio theory \cite{ferkar}. In this model, an $n$-dimensional diffusion consists of independent Brownian motions, except that the lowest-ranked particle receives an additional positive drift $\gamma$, inducing an attractive interaction. This mechanism leads to rich long-time behavior in both finite ($n<\infty$) and infinite ($n=\infty)$ dimensions; see, for instance, \cite{palpit, sar, sartsa, DJO, demtsa, banerjee2022domains, IPS, jourdain2008, shkolnikov2012large, jourdain2013propagation, andreis2019ergodicity}.

Related models motivated by evolutionary biology have also been extensively studied \cite{bebrnope, bebrnope2, berzha, bruder, durrem, atar2020hydrodynamics}. For example, \cite{durrem} considers a system in which particles give birth at rate one, with offspring positions sampled from a displacement kernel, and the leftmost particle is removed at each birth. Attractive topological interactions also arise in load-balancing models for queueing systems, where incoming jobs are routed toward shorter queues to maintain balance \cite{mitzenmacher2001power, lu2011join, der2022scalable, banerjee2023load}. Continuous-space pursuit and leader-follower models exhibiting attractive behavior have been analyzed in both probabilistic and physics contexts; see, e.g., \cite{banerjee2016brownian, ben2007toy, hongler2013local}. In parallel, a vast literature studies lattice-based attractive systems such as the voter model, contact process, and exclusion processes \cite{kipnis2013scaling}.

Flocking phenomena have been extensively studied through various models. \emph{Cucker-Smale models} quantify flocking for particle systems with continuous trajectories, where the velocities of particles update via interactions with those of other particles in the swarm \cite{10.1214/18-AAP1400}.
Recently, \emph{flocking for pure-jump models} have been probabilistically studied in \cite{Balzs2011ModelingFA,banerjee2024flockingfastlargejumps}, and by \cite{greenberg1996asynchronous,stolyar2023particle, stolyar2023large} in the context of distributed parallel simulation.  In all these models, particles interact with each other via sufficiently regular functionals of the empirical measure, like mean-based or quantile-based interactions. Such \emph{mean-field interactions} qualitatively differ from the \emph{nearest-neighbor interactions} present in our model, and they require very different techniques. In this sense, the Stochastic Follow-the-Leader system can be viewed as a 
{\em non-mean field, local interaction} model for flocking.

\subsection{Notation}
We denote a $k$-dimensional  vector $(x_1, \ldots x_k) \in \RR^k$ as $\bfx$. 
The set $\{1, \ldots , n\}$ will be denoted as $[n]$.
For a Polish space $\cls$,
$C_b(\cls)$ denotes the class of all real bounded continuous functions on $\cls$.  
The Borel $\sigma$-field on $\cls$ will be denoted as $\clb(\cls)$. The space of probability measures on $(\cls, \clb(\cls))$ will be denoted as $\clp(\cls)$ and equipped with the topology of weak convergence.
For a random variable $X$ with values in a Polish space $\cls$, 
$\mathcal{L}(X)$ denotes the probability law of $X$, and given a sub $\sigma$-field $\clg$, $\mathcal{L}(X \mid \clg)$ will denote the conditional law of $X$ given $\clg$ (which is a $\clg$-measurable $\clp(\cls)$-valued random variable). The statement that a random variable $X$ has the probability law $\theta$ will often be abbreviated as $X\sim \theta$.
For a Polish space $\cls$, we denote by $\mathcal{D}([0,\infty):\cls)$ the space of functions $f: [0,\infty) \rightarrow \cls$ which are right-continuous and have finite left-limits (RCLL) endowed with the usual Skorokhod topology. For a given $T>0$, the space $\mathcal{D}([0,T]:\cls)$ 
is defined similarly. For a function $f\in \cld([0,\infty):\cls)$, we denote its left limit at $t\in[0,\infty)$ as $f(t-)$.
 For two random variables $X$ and $Y$ with values in a Polish space $\cls$, $X\overset{d}{=}Y$ will denote that $X$ and $Y$ have the same probability law.
For $\mu$, $\nu\in \clp(\cls)$, we define the \emph{total variation distance} between them as $$\|\mu-\nu\|_{\text{TV}}:= \sup_{A\in \clb(\cls)}|\mu(A)-\nu(A)|.$$
For probability measures $\mu, \nu \in \clp(\RR)$ we use the notation $\mu \le_d \nu$ to denote that $\mu$ is stochastically dominated by $\nu$.
We will denote by $\mathbb{E}_W$ the expectation with respect to the law of the random variable $W$. Occasionally, if $\cll(W) = \theta$, we write such an expectation as $\mathbb{E}_{\theta}$.
We use $\lceil \cdot \rceil$ (resp. $\lfloor \cdot \rfloor$): $\RR_+ \to \NN_0$ to denote the smallest (resp. largest) integer greater than (resp. smaller than) or equal to a given nonnegative real number. The $n$-dimensional vector $(1, 1, \ldots 1)$ will be denoted as $\one$.
Exp$(\lambda)$ denotes the law of an Exponential random variable with mean $\lambda^{-1}$. 
The continuous uniform distribution over the interval $(a,b)$ will be denoted by U$(a,b)$.

\section{Model Description}\label{sec:model}
\noindent
Consider $n$ particles moving forward on the real line. The leading particle jumps forward at independent Exponentially distributed times with rate 1 and with each jump size being an independent $(0,\infty)$ valued random variable distributed as $\theta$.  The remaining particles also jump at independent Exponential times, with jump rates equal to their distance from the particle immediately ahead of them. Moreover, when any of these particles jump, they move to a random location chosen uniformly from the interval between the particle itself and the one immediately ahead of it.
$\Xbf(t)$ denotes the $n$-dimensional vector of the positions of the particles at time $t$. The particles are labelled in descending order- the leading particle is labelled as $X_1$, followed by $X_2$ and so on. Hence, for every $t\ge0$, $$X_n(t)<X_{n-1}(t)<\dots <X_2(t)<X_1(t).$$ 

In this work, our primary object of interest is the associated $(n-1)$-dimensional gap process. 
We denote the gap between the $i^{th}$ and $(i+1)^{th}$ particle at time $t$ as $$Y_i(t):= X_i(t)-X_{i+1}(t), \; i= 1, 2, \ldots, n-1.$$
Hence the gap process associated with the $n$-particle system is denoted at time $t\ge0$ by $$\Ybf(t) := (Y_1(t),Y_2(t),\dots,Y_{n-1}(t)).$$ 
Note that,  for $i\in [n]\setminus\{1\}$, given the state of the system up to some time $t>0$, the $i^{th}$ particle jumps with a rate of $Y_i(t-)$, from its current location, $X_i(t-)$, to its new location which is distributed as U$(X_i(t-),X_{i-1}(t-))$. Alternatively, we can say that its jump size, conditioned on $\sigma\{\Xbf(s): s < t\}$ and the event that a jump occurs at time instant $t$, is given by a U$(0,Y_{i-1}(t-))$-valued random variable. To keep notation and presentation simple we will drop the `$-$' (e.g. when writing $t-$) and qualifiers on conditional statements, such as `conditioned on the event that a jump occurs' in our informal descriptions when clear from the context.

For the remainder of this section, let $U$ denote a U$(0,1)$ random variable, sampled independently at each jump epoch. According to the dynamics described previously, the leading gap, denoted by $Y_1(t)$, may increase by a random variable distributed as $\theta$ at rate 1 (if the first particle jumps) or alternatively decrease by $Y_1(t)U$ at rate $Y_1(t)$ (if the second particle jumps). 
Note that, the new gap value also has a U$(0,Y_1(t))$ distribution. 

The remaining gaps behave similarly- the $i^{th}$ gap at time $t$, denoted by $Y_i(t)$ for $i\in [n-1]\setminus \{1\}$, can either increase by $Y_{i-1}(t)U$ at a rate of $Y_{i-1}(t)$ or decrease by $Y_{i}(t)U$ at a rate of $Y_{i}(t)$. 

Let $\cll_n$ denote the generator of the $\RR_+^{n-1}$ valued Markov process of the gaps, $\Ybf=(Y_1,Y_2,\dots,Y_{n-1})$, associated with the $n$-particle system. Then, for a bounded measurable map $f:\RR_+^{n-1}\to\RR$, 
\begin{align}\label{gen}
    \cll_nf(\textbf{y})&:= \Emb_U [f(\ybf-y_{n-1}U e_{n-1})-f(\ybf)] y_{n-1}\\
    &\quad +\sum_{i=1}^{n-2} \Emb_U [f(\ybf+y_iU (e_{i+1}-e_{i}))
    -f(\ybf)] y_i + \Emb_\theta[f(\ybf+Z e_1)-f(\ybf)],
\end{align}
where $Z$  denotes a random variable distributed as $\theta$, with cumulative distribution function $F_\theta$, and $U$ is a  U$(0,1)$ random variable, and $\Emb_{\theta}$ (respectively, $\Emb_U$) denotes expectation with respect to $Z$ (respectively, $U$). We will assume throughout that the mean of $\theta$ is finite. The infinite mean case will require some modifications to our techniques, although we expect the general approach to carry over to this case. To reduce parameters in the model, we also assume $\EE_\theta(Z) =1$; the case $\EE_\theta(Z) \neq 1$ can be treated similarly. 

Note that, we can rewrite \eqref{gen} as 
\begin{align}\label{gen.alt}
    \cll_nf(\ybf) &= \int_0^{y_{n-1}} f(\ybf -ue_{n-1})du\; + \sum_{i=1}^{n-2}\int_0^{y_i} f(\ybf+u(e_{i+1}-e_i))du\;\\
    &+\int_0^\infty f(\ybf+ue_1)\; F_\theta (du)\;- f(\ybf)\left(\sum_{i=1}^{n-1}y_i +1 \right).
\end{align}
Finally, let us consider the path space $\Omega=\mathcal{D}([0,\infty):\RR_+^{n-1})$ [resp. $\tilde \Om = \cld([0,\infty);\RR^n])$, $\mathcal{F}$ [resp. $\tilde \clf$] the corresponding Borel $\sigma$-field on $\Omega$ [resp. $\tilde \Om$]. On these two measurable spaces, we denote by $\PP_\ybf$ [resp. $\tilde \PP_\xbf$], the probability measures induced by $\Ybf$ [resp. $\Xbf$] when $\Ybf(0)=\ybf\in \RR_+^{n-1}$ [resp. $\Xbf(0)=\xbf \in \RR^n$]. We further denote by $\PP_\mu$, the probability measure induced by $\Ybf$ on $(\Om,\clf)$ when $\Ybf(0)\sim \mu \in \clp(\RR_+^{n-1})$. Abusing notation, canonical coordinate process on 
$(\Om, \clf)$ (resp. $(\tilde \Om, \tilde \clf)$) will again be denoted as $\Ybf$ (resp. $\Xbf$).
\section{Main Results}\label{sec:results}
Our first result in this paper concerns the stability of the gap process.
\begin{theorem}\label{thm:stat.ex}
    The Markov process of the gaps, $\Ybf = (Y_1, Y_2, ..., Y_{n-1})'$, has a unique stationary distribution $\pi_n$.
\end{theorem}

The proof is based on the construction of a suitable Lyapunov function. 

The next theorem gives uniform geometric ergodicity for the gap process, $\Ybf$. For this purpose, recall the Markov family $\{ \PP_\ybf\}_{\ybf\in \RR_+^{n-1}}$. For $t\in [0,\infty)$, consider the transition probability kernel of $\{ \PP_\ybf\}_{\ybf\in \RR_+^{n-1}}$ which can be defined as 
\begin{equation}\label{eq:640nn}
\PP^t(\ybf, A):= \PP_\ybf (\Ybf(t)\in A), \;\; t\ge 0,\; \ybf \in \RR_+^{n-1}, \; A\in \clb(\RR_+^{n-1}).
\end{equation}
 
 The following theorem will show that this transition probability kernel converges to the unique stationary distribution in the total variation distance at uniform geometric rates - that is, geometric rates independent of the initial configuration, $\ybf$.
\begin{theorem}\label{thm:unif.erg}
    There exists $\kappa\in (0,\infty)$ and $\beta\in (0,1)$, such that for every $t>0$ the following holds: $$\sup_{\ybf \in \RR_+^{n-1}} \| \PP^t(\ybf,\cdot)-\pi_n\|_\text{TV} \leq \kappa \beta^t.$$
\end{theorem}

The proof of uniform ergodicity proceeds through establishing a key minorization property and finiteness of exponential moments of a certain hitting time.

For our next few results, we consider the case when the leader's jump sizes follow an Exponential distribution. We prove that, in this case, the stationary distribution of the gaps takes an explicit form and is given by a product of Exponential laws as in the theorem below.
\begin{theorem}\label{thm:stat.exp.char}
Suppose that $\theta$ is the law of the Exponential random variable with rate $1$. Then the law of the stationary distribution $\pi_n$ of $\Ybf$ is given by Exp$(1)^{\otimes (n-1)}$. 
\end{theorem}

The proof relies on showing that,  the density corresponding to $\pi_n$ defined above solves $\cll_n^*\pi_n=0$, where $\cll_n^*$ is the adjoint of $\cll_n$ (see \eqref{gen.alt}).
\begin{remark}\label{rem:700n}
For $\lambda>0$, let $\hat{\Xbf}^{\lambda}(t) := \Xbf(\lambda^{-1}t)/\lambda$ and $\hat{\Ybf}^{\lambda}(t):= \Ybf(\lambda^{-1}t)/\lambda$. Suppose that, as in the above theorem, $\theta$ is the law of the Exponential random variable with rate $1$.
Then, by a simple scaling argument,  it is easy to verify that the Markov process $\hat{\Ybf}^\lambda$ has the unique stationary distribution Exp$(\lambda)^{\otimes (n-1)}$. Furthermore, the law of $\hat{\Xbf}^\lambda$ is the same as that of a variant of the original unscaled process where the leader takes Exp$(\lambda)$ jumps at rate $\lambda^{-1}$ and the subsequent particles behave exactly as in the original system. This observation provides insight into the mechanism underlying the one-parameter family of stationary distributions of the stick process described in Section \ref{sec:contr} (item 4).
\end{remark}

For the next result we again consider the setting where $\theta$ is Exp$(1)$. Our goal is to quantify the time required for the law at time 
$t$ of the gap process, started from an arbitrary initial distribution, to be suitably close to its stationary distribution. For this purpose, we work with the total variation distance between probability measures in $\clp(\RR_+^{n-1})$. For a given initial configuration (distribution) of the gaps, $\mu \in \clp(\RR_+^{n-1})$, let us denote the corresponding transition kernel for $t\in [0,\infty)$, analogous to \eqref{eq:640nn}  as $$\PP^t(\mu, A):= \PP_\mu(\Ybf(t)\in A),\;\; t\ge 0,\; \mu\in \clp(\RR_+^{n-1}),\; A\in \clb(\RR_+^{n-1}).$$  Define
\begin{equation}\label{eq:913n}
d(t):= \sup_{\mu \in \clp(\RR_+^{n-1})} \left\|\PP^t(\mu,\cdot)-\pi_n\right\|_\text{TV}\end{equation} and the \emph{mixing time} of the process $$t_\text{mix} (\eps) := \inf \{ t\geq 0: d(t)\leq \eps\}, \quad t_\text{mix}:= t_\text{mix}\left(\frac{1}{4}\right).$$ 
Using the contractive property of the total variation distance, we have for $l>0$, 
\begin{equation}\label{eq:911n}
d(lt_\text{mix})\le 2^{-l},
\end{equation}
and consequently, $$t_\text{mix}(\eps)\le \lceil\log_2 \eps^{-1}\rceil t_\text{mix}.$$
See  \cite[Equation 4.33-4.34]{levin2017markov}.
We now give upper and lower bounds for $t_\text{mix}$ in terms of the size of the particle system, $n$. 

\begin{theorem}\label{thm:tmix}
    Suppose that $\theta$ is Exp$(1)$. Then, there exist constants $c_1$, $c_2>0$, which are independent of $n$, such that the mixing time $t_\text{mix}$ of the gap process, $\Ybf$,  satisfies $$c_1 n\le t_\text{mix}\le c_2 n(\log n)^2.$$
\end{theorem}

To establish the lower bound, we use ideas from \cite[Section 7.3]{levin2017markov}. This approach relies on finding an appropriate {\em distinguishing statistic}, namely a map $\phi: \RR_+^{n-1} \to \RR$, such that the distance between the law of $\phi(\Ybf(t))$ and that of $\phi$ under stationarity can be bounded from below for a nonempty  collection of initial conditions, $ \mathcal{A}_n \subseteq \clp(\RR_+^{n-1})$.  Since, $d(t)$ is defined as the worst-case total variation distance (supremum over all $\mu$ in $\clp(\RR_+^{n-1})$), this will provide a lower bound on $t_\text{mix}$. 

In order to obtain an upper bound, we construct a coupling $(\Ybf_1,\Ybf_2)$ of the gap processes 
such that one of the two processes (say $\Ybf_2$) is started from stationarity - that is, $\Ybf_2(0) \sim \pi_n$, while the other ($\Ybf_1$) has an initial configuration with an arbitrary law $\mu$ in $\clp(\RR_+^{n-1})$. We denote this coupling
as $\Pbf_{\mu,\pi_n}$.
Namely, $\Pbf_{\mu,\pi_n}$ is the probability measure on $\cld([0,\infty): \RR_+^{2(n-1)})$ such that $\Pbf_{\mu,\pi_n}(\Ybf_1 \in \cdot) = \Pbf_{\mu}(\Ybf \in \cdot)$ and $\Pbf_{\mu,\pi_n}(\Ybf_2 \in \cdot) = \Pbf_{\pi_n}(\Ybf \in \cdot)$, where $\Ybf_1, \Ybf_2$ are $(n-1)$-dimensional coordinate processes on $\cld([0,\infty): \RR_+^{2(n-1)})$.
Define the \emph{coupling time} of $(\Ybf_1,\Ybf_2)$, denoted as $\tau_\text{coup}$, as 
\begin{equation}\label{eq:352n}
\tau_\text{coup}:= \inf \{ t\ge 0: \Ybf_1(s)=\Ybf_2(s) \text{ for all } s \ge t\}.\end{equation}
One can upper bound $d(t)$ using bounds on the tail probability of $\tau_\text{coup}$ (see \cite[Corollary 5.5]{levin2017markov}). Specifically, we use the fact that 
\begin{equation}\label{eq:tmix.up}
    d(t)\le \sup_{\mu \in \clp(\RR_+^{n-1})} \Pbf_{\mu,\pi_n}  \{ \tau_\text{coup}>t\}.
\end{equation}
Combining these two ideas provides us with the stated bounds on $t_\text{mix}$ in Theorem \ref{thm:tmix}. 

As a corollary of the above result and the observation made in Remark \ref{rem:700n} we have the following.  Recall the scaled processes (with $\lambda=n$)
$\hat{\Xbf}^n$ and $\hat{\Ybf}^n$ from Remark \ref{rem:700n}. 
\begin{corollary}\label{cor:fclt}
Let $t_n \in \RR_+$ be such that
{$\frac{t_n}{n^2(\log n)^2} :=\alpha_n \to \infty$} as $n\to \infty$. For $n \in \NN$, fix $\mu_n \in \clp(\RR_+^{n-1})$. Suppose that $X_1(0) =0$ and
$\bfY(0)$ is distributed as $\mu_n$. For $x \in [0,1]$, define
$$U^n(x):= \sqrt{n}\left(\hat X^n_1(t_n) - \hat X^n_{\lfloor nx\rfloor}(t_n) - x\right), \; x \in [0,1].$$
Then $U^n$ converges in distribution, in $\mathcal{D}([0,1]:\RR)$, to a standard Brownian motion.
\end{corollary}

Our final result concerns the case where the leader's jump sizes have a power law distribution. For this setting, we show that the gaps have lighter tails than that of the leader jump distribution, quantified by the finiteness of a higher moment. This result can be seen as a form of `regularization' under the stochastic follow-the-leader dynamics.
\begin{theorem}\label{thm:poly_tails}
    Suppose that, for some $k\ge 1$, $\EE_\theta Z^k<\infty$. Then, under  $\pi_n$, the $i^{th}$  gap, $Y_i(\cdot)$, for $i = 1, \ldots, n-1$, satisfies for any $t\ge 0$: 
    $$\EE_{\pi_n}[Y_i(t)]^{k+1}<\infty,$$ where $\EE_{\pi_n}$ denotes the expectation under the probability measure  $\PP_{\pi_n}$ and the latter measure is as introduced below \eqref{gen.alt}.
\end{theorem}
The proof relies on constructing suitable Lyapunov functions that quantify integrability properties of $\pi_n.$ 

%
\subsection{Organization}
The remainder of the paper is organized as follows.
In Section~\ref{sec:coupling} we introduce a coupling for the $n$-particle system, which serves as a key technical tool at multiple stages in the rest of the paper. Section~\ref{sec:stability} establishes the existence of a unique stationary distribution for the gap process (Theorem \ref{thm:stat.ex}) and proves uniform ergodicity (Theorem \ref{thm:unif.erg}). In
Section~\ref{sec:exp-case}, we analyze the Exponential leader case, obtaining the product
form stationary law (Theorem \ref{thm:stat.exp.char});  the mixing time bounds (Theorem \ref{thm:tmix}); and the functional limit theorem under a spatial scaling and large particle limit (Corollary \ref{cor:fclt}). Finally, Section~\ref{sec:heavy-tail} considers the
heavy-tailed case and derives the finiteness of higher moments for the gaps (Theorem \ref{thm:poly_tails}). 

\section{Coupling Construction}\label{sec:coupling}
\noindent In this section, we will introduce a coupling for the $n$-particle system. This will also automatically provide us with a coupling for the corresponding gap process. The ideas involved in the construction of the coupling and its properties discussed below will be used at multiple stages in the rest of the work.

Consider the coupling $(\textbf{X},\tilde \bfX)$ of $n$-particle systems started from $(\xbf,\tilde \xbf)$ described as follows. Since we are interested in the corresponding gap process for the particle system, we will only consider the case where $x_1=\tilde x_1$, namely,
the leading particles $X_1$ and $\tilde X_1$ start from the same position. Under our coupling, the leading particles evolve together using the same jump times and jump sizes, thus their locations will be equal at all time.

The dynamics of the remaining particles in the two systems is as follows. At any given time, consider the $(i+1)^{th}$ particle in each system, which jumps according to the size of the gap in front of it, namely, the $i^{th}$ gap. Recall that, the larger the gap, the higher the jump rate of the particle (faster). We introduce a coupling such that whenever the slower one of the two particles jumps, the faster one jumps as well. Moreover the jump sizes are such that the $i^{th}$ gaps now become equal. We refer to a jump of this type as a \emph{coalescence jump}. Besides this, the faster particle can make additional jumps - but the construction of the coupling ensures that it remains the \emph{faster} particle after the jump. That is, the larger one of the $i^{th}$ gaps remains larger until the coalescence jump takes place. We note that under this coupling, even if the $(i+1)^{th}$ particles undergo a coalescence jump, it does not ensure that they are now at the same position (since the $i$-th particle in the two systems may be at different location), or that the $i^{th}$ gaps will remain the same at all subsequent times (since the forward jumps of the $i$-th particle in the two systems may break this equality). However these two properties will be ensured if all the $i$ particles ahead of them have already coalesced, namely they are at the same locations in the two systems (and consequently would have the same states at all future times under the described coupling). Thus the system coalesces sequentially from the front to the back.  

We now give a precise description of the coupling construction. Let us consider, at any given time, the $(i+1)^{th}$ particles, $X_{i+1}(\cdot)$ and $\tilde X_{i+1}(\cdot)$, and denote the corresponding gaps (in front of them) by $Y_i(\cdot):= X_i(\cdot)-X_{i+1}(\cdot)$ and  $\tilde Y_i(\cdot):= \tilde X_i(\cdot)- \tilde X_{i+1}(\cdot)$. As stated earlier, we refer to the particle with the smaller (resp. larger) gap in front of it as the slower (resp. faster) particle, that is, if $Y_i(\cdot)<\tilde Y_i(\cdot)$, then $X_{i+1}(\cdot)$ is the slower particle and vice versa. Let us denote $m_i(\cdot)=Y_i(\cdot) \wedge \tilde Y_i(\cdot)$ and $n_i(\cdot)=Y_i(\cdot) \vee \tilde Y_i(\cdot)-Y_i(\cdot) \wedge \tilde Y_i(\cdot)$. We must have that the faster particle jumps at rate $m_i(\cdot) + n_i(\cdot)$, while the slower particle jumps at rate $m_i(\cdot)$. We devise the coupling in the following manner. At time $t$, the  jump for the  $(i+1)$-th pair of particles in the two systems occurs at rate $m_i(t)+n_i(t)$. The jump is one of the following two types:

\begin{enumerate}
    \item With probability $\frac{n_i(t)}{m_i(t)+n_i(t)}$, the faster particle jumps with the jump length distribution, U$(0,n_i(t))$, while the slower particle does not jump.
    \item \label{coup.jump} With probability $\frac{m_i(t)}{m_i(t)+n_i(t)}$, the faster particle jumps by an amount denoted by the random variable $U^*$, which is distributed as U$(n_i(t),m_i(t)+n_i(t))$. Furthermore, in this case, the slower particle simultaneously jumps by $U^*-n_i(t)$.
\end{enumerate}

Before we proceed, let us first establish that the scheme presented above ensures that $(\Xbf, \tilde \Xbf)$ serves as a valid coupling for the $n$-particle system described by the Markov family $\{\tilde{\PP}_\xbf\}_{\xbf\in \RR^{n}}$. For this purpose, it suffices to check that the marginal jump rates and jump size distributions of the particles in each system in the above coupling construction  match those in the particle system of interest. Again, consider the $(i+1)^{th}$ particles, $X_{i+1}(\cdot)$ and $\tilde X_{i+1}(\cdot)$, at a given time $t$. Without loss of generality, let us assume that, $Y_i(t)<\tilde Y_i(t)$ - that is, $X_{i+1}$ (resp. $\tilde X_{i+1}$) is the slower (resp. faster) particle at instant $t$ (Again, here and below, for simplicity of notation, we write $t$ in place of $t-$
and suppress the explicit conditioning events). Thus, we have $m_i(t)=Y_i(t)$ and $n_i(t)=\tilde Y_i(t)-Y_i(t).$ Hence, under the coupled dynamics described above, the jump rate of $X_{i+1}(t)$ is given by $\frac{m_i(t)}{m_i(t)+n_i(t)}(m_i(t)+n_i(t))=Y_i(t)$ and that of $\tilde X_{i+1}(t)$ is $m_i(t)+n_i(t)=\tilde Y_i(t)$. The jump size at instant $t$ (conditional on a jump occurring) of $X_{i+1}(\cdot)$ is given by $U^*-n_i(t)$ which is clearly distributed as U$(0, m_i(t))\text{, or equivalently }$U$(0,Y_i(t))$. Finally, the jump size of $\tilde X_{i+1}(\cdot)$, which we denote by $J$, can be described by $\chi U'+(1-\chi)U^* \sim \text{U}(0, \tilde Y_i(t))$, where $U'$ is a U$(0,n_i(t))$-valued random variable, $\chi$ is a Bernoulli random variable with probability of success given by $\frac{n_i(t)}{m_i(t)+n_i(t)}$, and $\chi, U', U^*$ are mutually independent. This verifies that under the above coupling construction $\bfX$ (resp. $\tilde{\bfX}$) has the distribution $\tilde{\PP}_\xbf$ (resp. $\tilde{\PP}_{\tilde{\xbf}}$).

Now, we will make some important observations about the coupling. First, note that, when the $i^{th}$ gaps in the two processes become equal, $X_{i+1}$ and $\tilde X_{i+1}$ jump together and by the same length which is given by a uniform random variable over the length of the gap. As explained before, we refer to a jump of type \ref{coup.jump} as a \emph{coalescence jump}. We say that the $(i+1)^{th}$ particles have \emph{coalesced} when for every $j\in [i]$, $X_j$ and $\tilde X_j$ are at the same position and then $X_{i+1}$ and $\tilde X_{i+1}$ undergo a coalescence jump. Note that after such a jump,
for every $j\in [i+1]$, $X_j$ and $\tilde X_j$ are the same at all future times.


Now, with $(\textbf{X},\tilde \bfX)$ constructed using the above coupling, consider the Markov process of the gaps, $(\Ybf, \tilde \Ybf)$, with initial values $(\Ybf(0), \tilde \Ybf(0))=(\ybf,\tilde \ybf)\in \RR_+^{2(n-1)}$. 
Since $(\xbf, \tilde \xbf)$ was arbitrary (other than $x_1= \tilde x_1$), $(\ybf, \tilde \ybf)$ can take any value in $\RR_+^{2(n-1)}$. 
We denote the corresponding joint law of $(\Ybf, \tilde \Ybf)$ as $\PP_{\ybf,\tilde \ybf}$ and the expectation as $\EE_{\ybf, \tilde \ybf}$. 
Note that the Markov process $(\Ybf, \tilde \Ybf)$ has the following generator: for bounded measurable $f:\RR_+^{2(n-1)}\to \RR$, 
\begin{multline}\label{coup.gen}
    \cll_n^{(\Ybf, \tilde \Ybf)} f(\ybf,\tilde \ybf) = \EE_\theta \left[ f(\ybf + Ze_1, \tilde \ybf +Ze_1) -  f(\ybf, \tilde \ybf) \right] \\
    + \sum_{i=2}^{n-2} y_i^* \EE_U \left[f(\ybf+y_i^* U (e_{i+1}-e_i)\textbf{1}_{y_i>\tilde y_i}, \tilde \ybf +y_i^* U (e_{i+1}-e_i)\textbf{1}_{y_i\leq \tilde y_i}) -  f(\ybf, \tilde \ybf) \right] \\
    + \sum_{i=2}^{n-2} y_i' \EE_U \left[f(\ybf+ (y_i^* \textbf{1}_{y_i>\tilde y_i} + y_i'U) (e_{i+1}-e_i), \tilde \ybf + (y_i^*  \textbf{1}_{y_i\leq \tilde y_i} + y_i'U )(e_{i+1}-e_i))-  f(\ybf, \tilde \ybf) \right] \\
     + y_{n-1}^* \EE_U \left[f(\ybf-y_{n-1}^* U e_{n-1}\textbf{1}_{y_{n-1}>\tilde y_{n-1}}, \tilde \ybf - y_{n-1}^* U e_{n-1}\textbf{1}_{y_{n-1}\leq \tilde y_{n-1}}) -  f(\ybf, \tilde \ybf) \right] \\
    + y_{n-1}' \EE_U \left[f(\ybf-(y_{n-1}^* \textbf{1}_{y_{n-1}  >\tilde y_{n-1}} + y_{n-1} U')e_{n-1}, \tilde \ybf - (y_{n-1}^*\textbf{1}_{y_{n-1}\leq \tilde y_{n-1}} + y_{n-1} U') e_{n-1}) -  f(\ybf, \tilde \ybf) \right], 
\end{multline}
where $y_i':= y_i \wedge \tilde y_i$, $y_i^* := (y_i \vee \tilde y_i) - y_i'$, $Z\sim \theta$ and $U\sim$U$(0,1)$. In the above formula, the first line captures the synchronous jump of the leading particles in the two systems; lines two and four the jumps of the faster particles; and lines three and five the coalescence jumps of the two particles. Note that each jump of a  particle, that is not the leader or the last, changes the gap process in two coordinates which leads to the  terms $e_{i+1}-e_i$ in the expression above.

\section{Stability of the Gap Process}\label{sec:stability}
In this section, we  will prove Theorems \ref{thm:stat.ex} and \ref{thm:unif.erg}. The first theorem shows that the gap process, $\Ybf$, has a unique stationary distribution, while the latter shows that the law of the process converges to stationarity,   in the total variation distance at a geometric rate, uniformly in the initial condition.
\subsection{Stationary Distribution - Existence and Uniqueness}
In this section we will prove the existence of a unique stationary distribution of the Markov family $\{\PP_\ybf\}_{\ybf \in \RR_+^{n-1}}$ (Theorem \ref{thm:stat.ex}). We begin with   the following result on the weak Feller property of   the gap process. Recall that a Markov process $\{\Zbf(t):t\ge 0\}$ on a Polish space $\cls$ is said to satisfy the weak Feller property if, for any $f\in C_b(\cls)$ and any $t \ge 0$, the function $\bfz \mapsto \EE_{\bfz}\left(f(\bfZ(t))\right)$ (with $\EE_{\bfz}$ denoting expectation under which  $\bfZ(0) = \bfz$ a.s.) is in $C_b(\cls)$.
Recall the coupling introduced in the last section for which the joint law of the $(n-1)$-dimensional gap process $(\Ybf, \tilde \Ybf)$ of the $n$-particle systems with $\mathbf{\Ybf}^n(0) = \bfy$ and $\tilde \Ybf^n(0) = \tilde \bfy$, was denoted by $\PP_{\bfy, \tilde \bfy}$.
\begin{lemma}\label{lem:wF}
Fix $n \in \NN$. 
For any $t \ge 0$, $\delta>0$ and $\bfy \in \RR_+^{n-1}$,
$$
\lim_{\tilde \bfy \rightarrow \bfy}\PP_{\bfy,\tilde \bfy}\left(\sum_{i=1}^{n-1} |Y^n_i(t) - \tilde Y^n_i(t)| \ge \delta\right)=0.
$$
In particular, the gap process of the $n$-particle system has the weak Feller property.
\end{lemma}
\begin{proof}
We will evaluate the generator in \eqref{coup.gen} applied to $f:\RR_+^{2(n-1)}\to \RR$ given by 
\begin{equation}\label{eq:1237n}
f(\xbf, \tilde \xbf )=\sum_{i=1}^{n-1} \left| x_i - \tilde x_i\right|, \; (\xbf, \tilde \xbf )\in \RR_+^{2(n-1)}.
\end{equation}

We will handle the five terms on the right-hand side of \eqref{coup.gen} (with the $y$'s there replaced by $x$'s) one by one. Clearly, the first term, given by $\EE_\theta \left[ f(\xbf + Ze_1, \tilde \xbf +Ze_1) -  f(\xbf, \tilde \xbf) \right]$ is 0. 

Now, for the second term, consider the $i^{th}$ component in the summation. For the $i^{th}$ component, after canceling the unchanged quantities, we are left with 
\begin{align*}
    x_i^*\Big[\EE_U|x_i-x_i^* U \textbf{1}_{x_i>\tilde x_i}-\tilde x_i + x_i^* U \textbf{1}_{x_i\leq\tilde x_i}|+\EE_U|x_{i+1}+x_i^* U \textbf{1}_{x_i>\tilde x_i}&-\tilde x_{i+1} - x_i^* U \textbf{1}_{x_i\leq\tilde x_i}|\\
    &-|x_i-\tilde x_i|-|x_{i+1}-\tilde x_{i+1}|\Big].
\end{align*}
Observe that, in the above expression, in the term inside $[\cdot]$, the difference between the first and third terms is $-x_i^*\EE_U(U)=-\frac{x_i^*}{2}$, while the difference between the second and the fourth terms is at most $x_i^*\EE_U(U)=\frac{x_i^*}{2}$. This makes the above expression non-positive. Thus, we can conclude that for our choice of $f$ the second term in \eqref{coup.gen} is non-positive, namely, $$\sum_{i=2}^{n-2} x_i^* \EE_U \left[f(\xbf+x_i^* U (e_{i+1}-e_i)\textbf{1}_{x_i>\tilde x_i}, \tilde \xbf +x_i^* U (e_{i+1}-e_i)\textbf{1}_{x_i\leq \tilde x_i}) -  f(\xbf, \tilde \xbf) \right]\leq 0.$$ 
Moving on to the third term in \eqref{coup.gen} and looking at the $i^{th}$ component of the sum, we use similar cancellation as before to see that it equals
\begin{multline}
    x_i'\Big[\EE_U|x_i-x_i^*  \textbf{1}_{x_i>\tilde x_i}- x_i' U-\tilde x_i  + x_i^* \textbf{1}_{x_i\leq\tilde x_i} +  x_i' U|\\
    +\EE_U|x_{i+1} +x_i^* \textbf{1}_{x_i>\tilde x_i}+ x_i' U-\tilde x_{i+1} - x_i^*  \textbf{1}_{x_i\leq\tilde x_i} - x_i' U|-|x_i-\tilde x_i|-|x_{i+1}-\tilde x_{i+1}|\Big].\label{eq:1232n}
\end{multline}
In the above, the first term in the expression inside $[\cdot]$ equals $0$, i.e., 
$$\EE_U|x_i-x_i^*  \textbf{1}_{x_i>\tilde x_i}- x_i' U-\tilde x_i  + x_i^* \textbf{1}_{x_i\leq\tilde x_i} +  x_i' U|=0.$$ For the second term note that, $\EE_U|x_{i+1}+x_i^* \textbf{1}_{x_i>\tilde x_i}+ x_i' U-\tilde x_{i+1} - x_i^*  \textbf{1}_{x_i\leq\tilde x_i} - x_i' U|$,  can be upper bounded by $|x_{i+1}-\tilde x_{i+1}|+x_i^*$. Combining this with the observation $x_i^*=|x_i-\tilde x_i|$ we see that the second term, and thus the expression in \eqref{eq:1232n}, is non-positive. Hence, with the above choice of $f$, the third term on the right side of \eqref{coup.gen} is non-positive as well, i.e. $$\sum_{i=2}^{n-2} x_i' \EE_U \left[f(\xbf+ (x_i^* \textbf{1}_{x_i>\tilde x_i} + x_i'U) (e_{i+1}-e_i), \tilde \xbf + (x_i^*  \textbf{1}_{x_i\leq \tilde x_i} + x_i'U )(e_{i+1}-e_i))-  f(\xbf, \tilde \xbf) \right] \leq 0.$$
The fourth term in \eqref{coup.gen}, can be simplified to the following expression $$x_{n-1}^* [|x_{n-1}-\tilde x_{n-1}|-x_{n-1}^*\EE_U(U)-|x_{n-1}-\tilde x_{n-1}|]= -\frac{(x_{n-1}^*)^2}{2}\leq0.$$
Finally the fifth term in \eqref{coup.gen} reduces to
$$x_{n-1}' [-|x_{n-1} -\tilde x_{n-1}|] \leq  0.$$ 
Combining these observations, we  see that, with $f$ as in \eqref{eq:1237n},
$$\cll_n^{(\Ybf, \tilde \Ybf)} f(\xbf, \tilde \xbf )\leq 0 \ \text{ for all } \ (\xbf, \tilde \xbf )\in \RR_+^{2(n-1)}.$$
We can assume without loss of generality that in the coupled system, the leading particle starts at $0$ and note that under our coupling $X_1(t) = \tilde X_1(t)$ for all $t$. We note that, for $M<\infty$, on the set $\{X_1(t) \le M\}$, $\{\bfY(s), \tilde{\bfY}(s),  0\le s \le t\}$ take values in a bounded set, $\PP_{\bfy,\tilde \bfy}$ a.s. Letting $\tau_M = \inf\{t \ge 0: X_1(t) \ge M\}$, we have by Dynkin's formula
\begin{align}
\EE_{\bfy,\tilde \bfy}\left(\sum_{i=1}^{n-1} \left|Y_i^n(t \wedge \tau_M) - \tilde Y_i^n(t\wedge \tau_M)\right|\right) 
\le 
\sum_{i=1}^{n-1} |y_i - \tilde y_i|.\label{eq:150nn}
\end{align}
In obtaining the above inequality we have used the fact that the  carr\'{e}-du-champ operator associated with $\cll_n^{(\Ybf, \tilde \Ybf)}$ on the function $f$ is bounded on $\{(\bfY(s), \tilde{\bfY}(s)): 0\le s < \tau_M\}$ and thus the local martingale in the Dynkin's formula is in fact a martingale.

Thus, we conclude that, for every $t\geq 0$,
\begin{align*}
\PP_{\bfy,\tilde \bfy}\left(\sum_{i=1}^{n-1} \left|Y_i^n(t) - \tilde Y_i^n(t)\right| \ge \delta\right) &\le \PP_{\bfy,\tilde \bfy}\left(\sum_{i=1}^{n-1} \left|Y_i^n(t) - \tilde Y_i^n(t)\right| \ge \delta, \; X_1(t) \ge M\right)\\
&\quad +
\delta^{-1}\EE_{\bfy,\tilde \bfy}\left(\sum_{i=1}^{n-1} \left|Y_i^n(t\wedge \tau_M) - \tilde Y_i^n(t\wedge \tau_M)\right|\right) \\
&\le 
\PP_{\bfy,\tilde \bfy}(X_1(t) \ge M)+
\delta^{-1}\sum_{i=1}^{n-1} |y_i - \tilde y_i|,
\end{align*}
where the first inequality uses Markov's inequality.
The first statement in the lemma is now immediate on sending $M\to \infty$ and $\tilde{\bfy} \to \bfy$. The second statement in the lemma is a straightforward consequence of the first.
\end{proof}
We now prove the existence and uniqueness of stationary distributions for the gap process $\Ybf.$
\subsubsection{Proof of Theorem \ref{thm:stat.ex}}
 The proof is based on constructing an appropriate Lyapunov function. For $i\in [n-1]$,  consider the functions, $q_i:\RR_+^{n-1}\to \RR$ defined as $q_i(\textbf{y})=y_i$, $\bfy \in \RR_+^{n-1}$. 
Using \eqref{gen}, we get, for $\bfy \in \RR_+^{n-1}$,
\begin{equation}\label{eq:q1}
    \cll_n q_1(\textbf{y}) = \Emb_\theta(Z) + y_1\Emb_U(-y_1 U)  = 1-\frac{y_1^2}{2},
\end{equation}
where $Z\sim \theta$ and $U\sim \text{U}(0,1)$. 

Further, for $i\in [n-1]\setminus \{1\}$, 
\begin{align}\label{eq:qi}
    \cll_n q_i(\textbf{y})&= y_{i-1}\Emb_U(y_{i-1}U)+y_{i}\Emb_U(-y_{i}U) = \frac{y_{i-1}^2-y_i^2}{2}.
\end{align}
Now, let us consider the {\em Lyapunov function}, $V:\RR_+^{n-1}\to \RR_+$ defined as 
\begin{equation}\label{eq:lyapdef}
V(\textbf{y})=1+\sum_{i=1}^{n-1} \alpha^i y_i =1+\sum_{i=1}^{n-1} \alpha^i q_i(\ybf), \; 
\bfy \in \RR_+^{n-1},
\end{equation}
where we take $\alpha= 10^{-1}.$ 
Applying the generator $\cll_n$ to $ V$, we get, \begin{align*}
    \cll_n V(\ybf)&=\alpha \left(1- \frac{y_1^2}{2} \right) +\sum_{i=2}^{n-1} \alpha^i \left( \frac{y_{i-1}^2-y_i^2}{2}\right) \\
    &= \alpha-\sum_{i=1}^{n-2} \left( \frac{\alpha^i-\alpha^{i+1}}{2}\right)y_i^2 - \frac{\alpha^{n-1}}{2}y_{n-1}^2\\
    &= \alpha-\sum_{i=1}^{n-1}\left( \frac{\alpha^i-\alpha^{i+1}}{2}\right)y_i^2-\frac{\alpha^n}{2}y_{n-1}^2.
\end{align*}
Hence, we  conclude that, \begin{equation}\label{drift1}
    \cll_n V(\ybf)\leq \alpha - \left( \frac{1-\alpha}{2}\right)\sum_{i=1}^{n-1}\alpha^i y_i^2.
\end{equation}%
Defining the compact set $K:= \{ \ybf\in \RR^{n-1}: \sum_{i=1}^{n-1} \alpha^iy_i^2 \leq \frac{4\alpha}{1-\alpha}\}$, we have 
\begin{equation}\label{eq:111n}
\cll_nV(\ybf)\leq - \alpha + 2\alpha\mathbf{1}_K(\ybf), \;\; y\in \RR_+^{n-1}.
\end{equation}
This, along with the weak Feller property of the Markov process $\Ybf(\cdot)$ deduced in Lemma \ref{lem:wF}, says that there must be at least one  stationary distribution $\pi_n$ for the process (cf.  \cite[Corollary 1.18]{eberle2015markov}). 

Next, we argue that the stationary distribution $\pi_n$ is unique. It suffices to construct a coupling $(\bfY, \tilde \bfY)$ of the Markov process under consideration, with law denoted by $\PP_{\bfy,\tilde \bfy}$ when started from $(\bfy, \tilde \bfy)$, such that the associated coupling time $\tau_\text{coup}$ defined in \eqref{eq:352n} satisfies $\PP_{\bfy,\tilde \bfy}(\tau_\text{coup} < \infty)>0$ for all $(\bfy, \tilde \bfy) \in \RR_+^{2(n-1)}$. See  \cite[Lemma 6.2]{banerjee2024flockingfastlargejumps} for a proof of this fact (the measurability condition required in the latter result will be immediate from the construction).

We recall the coupling constructed in Section \ref{sec:coupling} and the fact that the coupling for the particle systems given by ($\Xbf, \tilde \Xbf$) naturally produces a coupling for the gap process, which we denote by $(\Ybf, \tilde \Ybf)$. Observe that the construction of the coupling ensures that for $i \ge 2$, if the particles $X_j$ are all coalesced with $\tilde X_j$ for $j \le i-1$, and the slower particle in the pair $(X_i, \tilde X_i)$ jumps, then $X_i$ and $\tilde X_i$ jump to the same location. Therefore, starting from the pair $(X_2,\tilde X_2)$ (recall that $(X_1,\tilde X_1)$ start and move together) and coalescing the successive particles sequentially ensures that the coupling time for the processes is finite with positive probability. To state it more precisely, for the described coupling $(\Xbf, \tilde \Xbf)$, consider the event where the first $(n-1)$ jumps of the $2n$-dimensional process are given, in order, by a coalescence jump for $(X_2,\tilde X_2)$ followed by a coalescence jump for $(X_3, \tilde X_3)$ and so on. Then, it is clear that after the first $(n-1)$ jumps, the states of $\Xbf$ and $\tilde \Xbf$ (resp. $\Ybf$ and $\tilde \Ybf$) become the same and stay equal at all future times. As this event happens with positive probability, we conclude that $\PP_{\bfy,\tilde \bfy}(\tau_\text{coup} < \infty)>0$. The uniqueness of $\pi_n$ follows.
\hfill
\qedsymbol

\subsection{Uniform Ergodicity}
To establish uniform ergodicity, we will require two key ingredients. First, a \emph{minorization property} is derived which provides a suitable lower bound on the transition probabilities of the gap process on any nonempty compact  set $C_0$. Next, we obtain bounds on the \emph{exponential moments} of the hitting times of a particular choice of a compact set $C$, \emph{uniformly in the initial configuration}. The latter will be accomplished by making use of the drift property established in \eqref{drift1}. The set $C$ we choose
is defined in terms of the Lyapunov function introduced in \eqref{eq:lyapdef}, as
\begin{equation}\label{eq:628n}
C := \{ \ybf \in \RR_+^{n-1}: V(\ybf)\leq 4\}.\end{equation}
We denote the associated hitting time of the set  as $\tau_{C} := \inf \{ t\geq0: \Ybf(t) \in C\} .$ 
In the following lemma, we derive the minorization estimate that we require. In the terminology of Meyn and Tweedie\cite{meytwe}, this says that any compact set $C_0$ is a {\em small} set for this Markov process.

\begin{lemma}\label{lem:minor.cnd}
Fix a nonempty compact set $C_0 \subset \RR_+^{n-1}$. There exists $\epsilon>0$ and $\nu\in \mathcal{P}(\RR_+^{n-1})$ such that for every $t >0$, there is a $c(t)>0$ such that for all $\ybf\in C_0$,
    \begin{equation}\label{minor}
        \PP^t (\ybf, A)\geq \epsilon c(t)\nu(A),\;\;\text{ for every } A\in \clb(\RR_+^{n-1}).
    \end{equation}
    Furthermore, $\inf_{t \in [a,b]} c(t)>0$ for every $0< a \le b <\infty$.
\end{lemma}
\begin{proof}
To prove the above for every Borel set in $\RR_+^{n-1}$, we will prove it for a ``nice" sub-class of sets and then complete the proof using a monotone class argument. 
Define the class of sets, 
\begin{align*}
\cls^\circ &:= \left\{ \prod_{i=1}^{n-1}[c_i,a_i):\; 0\leq c_{i}<a_{i}\le\infty \right\},\\
\cls &:= \left\{ \cup_{j=1}^m S_j : m\in\NN,  S_j \in \cls^\circ \text{ for all } 1 \le j \le m, S_i \cap S_j = \emptyset \text{ for } i \neq j\right\} \cup\{\emptyset\}.
\end{align*}
Denote by $\clg= \clg(\epsilon)$ the collection of all subsets of $\RR_+^{n-1}$ that satisfy the minorization condition \eqref{minor} for all $\ybf \in C_0$, for a given value of $\epsilon$. Since $\clg$ is a monotone class that is closed under disjoint unions, and $\cls$ is an algebra formed from finite disjoint unions of sets in $\cls^\circ$, it is enough to prove that, for some $\epsilon>0$, the minorization condition \eqref{minor} holds, with a suitable $c(t)$, for all sets in $\cls^\circ$, $\ybf \in C_0$, and $t >0$. Take any $\ybf \in C_0$ and consider $H:= \prod_{i=1}^{n-1} [c_{i},a_{i}) \in \cls^\circ$ and  $t >0$. We will construct an event contained in $\{\Ybf(t)\in H\}$ and  use the probability of this event,
under $\PP_{\ybf}$,
to obtain a lower bound for $\PP^t(\ybf, H).$ Then, we will choose $\epsilon$ and $\nu(\cdot)$ appropriately, independent of $t$, so that the desired minorization condition holds.

Recall the jump size distribution of the leading particle is denoted as $\theta$. Select $\lambda>0$, such that  $\theta_{\lambda,n} := \theta([\lambda,\lambda(1+\frac{1}{n-1})])>0$. Fix $t >0$ and partition the time interval $[0,t]$ into $n$ intervals of equal length: $\{\left[t_{i-1}, t_{i}\right] : 1 \le i \le n\}$, where $t_i := \frac{it}{n}$. We will establish \eqref{minor} with $\nu$  as the uniform measure on $[0,\lambda]^{n-1}$. Thus, in establishing this estimate for a set $H=\prod_{i=1}^{n-1} [c_{i},a_{i})$, it suffices to consider the case where $a_{i} \le \lambda$ for every $i\in[n-1]$.

 The event  $B$ contained in $\{\Ybf(t)\in H\}$ that we will use is defined as follows. On $B$, 
 \begin{itemize}
 \item over the time interval $[t_0,t_1]$,  the leading particle $X_1$ jumps exactly $(n-1)$ times, each time by a length of at least $\lambda$ and at most $\lambda\left(1+\frac{1}{n-1}\right)$.
 \item  All the other particle states remain unchanged over this interval.
 \item Over the time interval $\left[ t_{i-1},t_i \right]$, $i\in[n]\setminus\{1\}$, all particles states, except the $i^{th}$ particle,  remain unchanged. 
 \item On this interval $X_{i}$ jumps exactly once   and its jump length is such that, after the jump, the $(i-1)^{th}$ gap, $Y_{i-1}(\cdot)\in [c_{i-1}, a_{i-1})$. 
 \end{itemize}
 Note that $B \subset \{\Ybf(t)\in H\}$. 
 We write $B=\bigcap_{i=1}^{n} B_i$ where $B_i$ is the event corresponding to the above requirements over the time interval $\left[ t_{i-1},t_i\right]$ (in the first two bullet items for $i=1$, and  in the last two bullet items for $i>1$). 
 Then, $\PP_\ybf(B)= \PP_\ybf(B_1)\prod_{i=2}^{n} \PP_\ybf\left(B_i| B_1, \dots, B_{i-1}\right)$. 
 
 Since $C_0$ is compact, there is a $\alpha>0$ such that  $\sum_{j=1}^{n-1}  y_j \leq \alpha$ for all $\bfy \in C_0$. 
Since, for each $i = 1,\ldots,n$, on the event $\bigcap_{j=1}^i B_j$ the leading particle makes  $(n-1)$ jumps over the time interval $[0,t_i]$, each of size no larger than $\lambda\!\left(1+\frac{1}{n-1}\right)$, it follows that, on $\bigcap_{j=1}^i B_j$,
\begin{equation}\label{eq:contr}
    \sum_{i=1}^{n-1} Y_i(s)\leq \alpha+n\lambda,\; \text{ for every } s\in[0,t_i].
\end{equation}
Let us now obtain a lower bound for $\PP_\ybf(B_1)$. The probability of the first particle behaving as described in the interval $[t_0,t_1]$ is given by $e^{-t n^{-1}}\left(\frac{t}{n}\right)^{n-1}\frac{1}{(n-1)!}(\theta_{\lambda,n})^{n-1}$. The probability that $X_2$ does not jump in this time interval is at least $e^{-tn^{-1}(y_1+n\lambda)}$. Further, given the particle $X_2$ does not move, the probability that none of the other particles jump in $[t_0,t_1]$ is lower bounded by $e^{-tn^{-1}\sum_{i=2}^{n-1}y_i}$. Hence, using \eqref{eq:contr} (with $i=1$), we get $$\PP_\ybf(B_1)\geq e^{-t n^{-1}(1+\sum_{i=1}^{n-1}y_i+n\lambda)}\left(\frac{t}{n}\right)^{n-1}\frac{(\theta_{\lambda,n})^{n-1}}{(n-1)!}\geq e^{-t n^{-1}(1+\alpha+2n\lambda)}\left(\frac{t}{n}\right)^{n-1}\frac{(\theta_{\lambda,n})^{n-1}}{(n-1)!}.$$  

Now we consider $i\in[n]\setminus\{1\}$. 
The probability that the $i^{th}$ particle jumps exactly once in $[t_{i-1},t_i]$, given that the particles ahead of it do not move over this interval, is at least $e^{-t n^{-1}Y_{i-1}(t_{i-1})}Y_{i-1}(t_{i-1})\frac{t}{n}$. Here we use the fact that,  on the event that particles ahead not move, the jump rate of a particle is non-increasing in time.
Note that on $\cap_{j=1}^{i-1}B_j$, $Y_{i-1}(t_{i-1}) > \lambda$.
Further, conditioned on $\cap_{j=1}^{i-1}B_j$ and that the particles ahead of $X_i$ do not move on $[t_{i-1},t_i]$, the probability that in the single jump, the gap $Y_{i-1}(\cdot)$ shrinks to a value in the interval $[c_{i-1}, a_{i-1})$ is given by $\frac{a_{i-1}-c_{i-1}}{Y_{i-1}(t_{i-1})}$. When $i\neq n$, the probability that the $(i+1)^{th}$ particle does not jump in this time interval, given the particles ahead of the $i^{th}$ particle do not move, is at least $e^{-t n^{-1}(Y_i(t_{i-1})+Y_{i-1}(t_{i-1}))}$. This is because, on the event the particles ahead of the $i^{th}$ particle do not move, the jump rate of the $(i+1)^{th}$ particle increases after the $i^{th}$ particle jumps. For any $j \neq 1,i,i+1$, the conditional probability that the $j^{th}$ particle does not move on this interval, given the ones ahead are static all through the interval, equals $e^{-t n^{-1}\left(Y_{j-1}(t_{i-1})\right)}$. The probability of the first particle not jumping on this interval is simply $e^{-t n^{-1}}$.
Combining these observations, we get 
\begin{align}
    &\PP_\ybf(B_i| B_1, \dots,B_{i-1})\\
    &\geq e^{-t n^{-1}\left(1+\sum_{j=1}^{n-1}Y_j(t_{i-1})\right)} e^{-t n^{-1}Y_{i-1}(t_{i-1})} Y_{i-1}(t_{i-1}) \frac{t}{n} \frac{[a_{i-1}-c_{i-1}]}{Y_{i-1}(t_{i-1})} \\
    &= e^{-t n^{-1}\left(1+\sum_{j=1}^{n-1}Y_j(t_{i-1})+Y_{i-1}(t_{i-1})\right)}  \frac{[a_{i-1}-c_{i-1}]t}{n}. 
\end{align}
Using \eqref{eq:contr}, we conclude that $$\PP_\ybf(B_i| B_1, \dots,B_{i-1}) \geq e^{-t n^{-1}\left(1+2(\alpha+n\lambda)\right)}  \frac{[a_{i-1}-c_{i-1}]t}{n}.$$ 
Recalling that $B\subset H$, we obtain 
\begin{align}
    \PP^t(\ybf,H)\geq \PP_\ybf(B) &=  \PP_\ybf(B_1)\prod_{i=2}^{n} \PP_\ybf(B_i| B_1, \dots,B_{i-1}) \\
    &\geq e^{-t\left(1+2(\alpha+n\lambda)\right)}\left(\frac{t}{n}\right)^{n-1}\frac{(\theta_{\lambda,n})^{n-1}}{(n-1)!} \prod_{i=2}^{n} \left( \frac{[a_{i-1}-c_{i-1}]t}{n}\right).
\end{align}
Thus, letting $$\epsilon = \frac{(\theta_{\lambda,n})^{n-1}}{(n-1)!}n^{-2(n-1)}\lambda^{n-1}, \; c(t) = e^{-t\left(1+2(\alpha+n\lambda)\right)} \left(\frac{t}{n}\right)^{2(n-1)},
$$ and taking $\nu$ to be the uniform measure on $[0,\lambda]^{n-1}$, i.e. 
$$\nu(H) = \nu\left(\prod_{i=1}^{n-1}[c_{i},a_{i})\right)=\lambda^{-(n-1)}\prod_{i=1}^{n-1} (a_{i}-c_{i}),
$$ 
for sets $H= \prod_{i=1}^{n-1}[c_{i},a_{i}) \subseteq [0,\lambda]^{n-1}$ in $\cls^\circ$, we have that the minorization condition, $\PP^t(\ybf, H)\geq \epsilon c(t) \nu(H)$, holds for every $t >0$, $\ybf \in C_0$ and $H \in \cls^\circ$. 
Also note that for all $0< a\le b<\infty$, 
$\inf_{t \in [a,b]} c(t) >0$.
The result follows.
\end{proof}

We will now consider the compact set $C$ introduced in \eqref{eq:628n} and establish the finiteness of exponential moments for hitting times of $C$, uniformly over all initial configurations.
\begin{lemma}\label{lem:exp.mom.bd}
    There exists $\eta\in (0, \infty)$ such that $$\sup_{\ybf \in \RR_+^{n-1}}\EE_{\ybf} e^{\eta \tau_{C}}<\infty.$$
\end{lemma}
\begin{proof}
    The proof proceeds by using the drift property \eqref{drift1} of the Lyapunov function $V$ defined in \eqref{eq:lyapdef}   which provides bounds for the tail probability of $\tau_C$ using an associated exponential supermartingale.
    
    We begin by observing a modification of the drift inequality in \eqref{drift1} that better suits our purposes. By the Cauchy-Schwarz inequality, for $\bfy \in \RR_+^{n-1}$,  $$\left(\sum_{i=1}^{n-1}\alpha^i y_i \right)^2 \leq \left(\sum_{i=1}^{n-1} \alpha^i \right) \cdot \left( \sum_{i=1}^{n-1}\alpha^i y_i^2 \right).$$ Observing that, for every $n\in \NN$, $\sum_{i=1}^{n-1} \alpha^i < \sum_{i=1}^{\infty} \alpha^i = \frac{\alpha}{1-\alpha},$ we have $$\sum_{i=1}^{n-1} \alpha^i y_i^2 \geq \left(\frac{1-\alpha}{\alpha} \right)\cdot \left(\sum_{i=1}^{n-1}\alpha^i y_i \right)^2 = \left(\frac{1-\alpha}{\alpha} \right) \left[ V(\ybf)-1\right]^2. $$ 
    Combining this with \eqref{drift1}, we now have  the following modification of the drift inequality, 
\begin{equation}\label{eq:1013}
    \cll_n V(\ybf)\leq \alpha -  \frac{(1-\alpha)^2}{2\alpha} \left[ V(\ybf)-1\right]^2=\alpha -  \alpha' \left[ V(\ybf)-1\right]^2, \; \bfy \in \RR_+^{n-1}, 
\end{equation}
where $\alpha' := \frac{(1-\alpha)^2}{2\alpha}$.

Now let us define a sequence of stopping times $\tau_i:= \inf\{t\geq 0: V(\Ybf(t))\leq \lambda_i\}$ for $i\in \NN$, where $\lambda_i=4^i$. We also define $\tau_\infty:= 0.$
Note that, for every $\bfy \in \RR_+^{n-1}$, $\tau_{C}=\tau_1 =  \sum_{i=1}^\infty (\tau_i -\tau_{i+1})$, $\PP_{\bfy}$ a.s. 
Thus, for any $t\geq0$,
\begin{align}
    \PP_\ybf(\tau_C\geq t)&=\PP_\ybf\left(\sum_{i=1}^{\infty}(\tau_i-\tau_{i+1}) \geq t\right).
\end{align}
Define  $\beta_i:= 2^{-i}$, for $i\in \NN$. 
Since $\sum_{i=1}^\infty \beta_i=1$,
\begin{align}\label{eq:psum}
    \PP_\ybf(\tau_C\geq t) &= \PP_\ybf\left(\sum_{i=1}^{\infty}(\tau_i-\tau_{i+1}) \geq t\sum_{i=1}^{\infty}\beta_i \right) \\
    &\leq \sum_{i=1}^{\infty} \PP_\ybf\left(\tau_i-\tau_{i+1} \geq t\beta_i \right).
\end{align}
For  $\theta>0$ that will be chosen suitably later on, let us define the process, $\zeta(t):= e^{\theta t}V(\Ybf(t)).$ Using Dynkin's formula we have that
$$M(t) := \zeta(t)-V(\ybf)-\int_0^t \left[e^{\theta s}\cll_n V(\Ybf(s))+\theta e^{\theta s}V(\Ybf(s))\right]\;ds, \; t \ge 0$$
is a $\clf_t$- local martingale under $\PP_\ybf$ for every $\ybf \in \RR_+^{n-1}$, where
$\clf_t := \sigma\{\Ybf(s): s \le t\}$.

Note from \eqref{eq:1013} that, for $s\in [t \wedge \tau_{i+1}, t\wedge \tau_i)$, 
\begin{align}
    \cll_nV(\Ybf(s))&\le \alpha - \alpha' - \alpha' V(\Ybf(s))^2+2\alpha' V(\Ybf(s)) \\
    &\leq \alpha - \alpha' +\alpha'(  2-  \lambda_i) V(\Ybf(s)).
\end{align}
Thus, we have 
\begin{align}
    \cll_n V (\Ybf(s)) + \theta  V(\Ybf(s)) \le \alpha - \alpha' +\alpha'(  2+\theta-  \lambda_i) V(\Ybf(s)).
\end{align}
Since $\alpha = 10^{-1}$, we have $\alpha-\alpha' <0.$ Now letting $\theta = \frac{\lambda_i}{2}$, we have 
\begin{equation}\label{eq:1022n}
\cll_n V (\Ybf(s)) + \theta  V(\Ybf(s)) \le 0,\; 
s\in [t \wedge \tau_{i+1}, t\wedge \tau_i), \; i\in \NN.
\end{equation}
From the local martingale property of $M$ and \eqref{eq:1022n}, we have
$$
\EE_{\bfy}\left[
\left(e^{\frac {1}{2}\lambda_i(\tau_i \wedge t)} V(\bfY(\tau_i\wedge t)) - 
e^{\frac {1}{2}\lambda_i(\tau_{i+1} \wedge t)} V(\bfY(\tau_{i+1}\wedge t))\right) \mid \clf_{\tau_{i+1}\wedge t}
\right] \le 0
$$
and consequently, since $V \ge 1$,
\begin{equation}\label{eq:1106}
\EE_{\bfy}\left[
e^{\frac {1}{2}\lambda_i(\tau_i \wedge t - \tau_{i+1} \wedge t)}  \mid \clf_{\tau_{i+1}\wedge t}
\right] \le \EE_{\bfy}\left[
e^{\frac {1}{2}\lambda_i(\tau_i \wedge t - \tau_{i+1} \wedge t)} V(\bfY(\tau_i\wedge t))  \mid \clf_{\tau_{i+1}\wedge t}
\right] \le V(\bfY(\tau_{i+1}\wedge t)), \; \PP_{\bfy} \mbox{ a.s. }
\end{equation}
Here we have used the fact that the
carr\'{e}-du-champ term $\cll_n(V^2)(\bfy)- 2V(\bfy)\cll_n V(\bfy)$ is bounded on compact subsets of $\RR_+^{n-1}$ and appealed to an argument similar to the one used in \eqref{eq:150nn}.

We now argue that, for each $\bfy \in \RR_+^{n-1}$,
$\PP_{\bfy}$ a.s., 
\begin{equation}\label{eq:1044n}
\tau_i < \infty \mbox{ and } \EE_\ybf\left(e^{\frac {1}{2}\lambda_i(\tau_i - \tau_{i+1})}\right) \le \lambda_{i+1}, \; \mbox{ for every $i \in \NN$}.
\end{equation}
The proof of \eqref{eq:1044n} is completed using a recursive argument. Fix $\bfy \in \RR_+^{n-1}$ and choose $i^*$ sufficiently large so that $V(\bfy) \le \lambda_i$ for all $i \ge i^*+1$. Then, 
\eqref{eq:1044n} holds for $i\ge i^*+1$.
Now suppose that \eqref{eq:1044n} holds for some $i$ with $i^*+1\ge i>1$ and consider $i-1$, and assume without loss of generality that $V(\bfy) > \lambda_i$ (otherwise the previous argument applies). Take any $t>0$. From 
\eqref{eq:1106},
$$
\EE_{\bfy}\left[
e^{\frac {1}{2}\lambda_{i-1}(\tau_{i-1} \wedge t - \tau_{i} \wedge t)} \right] \le \EE_{\bfy}V(\bfY(\tau_{i}\wedge t)).
$$
Moreover, as $V(\Ybf(s)) > \lambda_i$ for $0\le s < \tau_i$, by \eqref{eq:1013}, $\cll_n V (\Ybf(s)) <0$ for all $s \in [0, \tau_i \wedge t)$.
This, together with \eqref{eq:111n}, Dynkin's formula, and monotone convergence shows that
$$
\EE_{\bfy}\left[
e^{\frac {1}{2}\lambda_{i-1}(\tau_{i-1} \wedge t - \tau_{i} \wedge t)} \right] \le \EE_{\bfy}V(\bfY(\tau_{i}\wedge t)) \le 
V(\bfy).$$
Using monotone convergence we now have, on sending $t\to \infty$,
$$
\EE_{\bfy}\left[
e^{\frac {1}{2}\lambda_{i-1}(\tau_{i-1}  - \tau_{i})} \right] \le V(\bfy).
$$
In particular, $\tau_{i-1} < \infty$, $\PP_{\bfy}$ a.s.
This also shows that 
$$e^{\frac {1}{2}\lambda_{i-1}(\tau_{i-1} \wedge t - \tau_{i} \wedge t)} \to e^{\frac {1}{2}\lambda_{i-1}(\tau_{i-1}  - \tau_{i})} \mbox{ in } L^1(\PP_{\bfy}).$$
This together with martingale convergence theorem and \eqref{eq:1106} shows that
\begin{align*}
\EE_{\bfy}\left[
e^{\frac {1}{2}\lambda_i(\tau_{i-1} - \tau_{i})}  \mid \clf_{\tau_{i}-}
\right] &= \lim_{t\to \infty}
\EE_{\bfy}\left[
e^{\frac {1}{2}\lambda_i(\tau_{i-1} \wedge t - \tau_{i} \wedge t)}  \mid \clf_{\tau_{i}\wedge t}
\right]  \le \lim_{t\to \infty} V(\bfY(\tau_{i}\wedge t)) = V(\bfY(\tau_{i})) \le \lambda_i,
\end{align*}
where $\clf_{\tau_{i}-} := \sigma\left(\cup_{t \ge 0}\clf_{\tau_{i}\wedge t}\right)$.
Taking expectations we have
$$\EE_{\bfy}\left[
e^{\frac {1}{2}\lambda_i(\tau_{i-1} - \tau_{i})}\right] \le \lambda_i$$
completing the recursion step and proving \eqref{eq:1044n}.





Using Markov's inequality we now have, for $i\ge 1$
\begin{align}
    \PP_\ybf\left(\tau_i - \tau_{i+1}\geq t\right) \leq {e^{-\frac{1}{2}\lambda_it }}\EE_\ybf\left(e^{\frac {1}{2}\lambda_i(\tau_i - \tau_{i+1})}\right) \leq \lambda_{i+1}e^{- \frac{1}{2}\lambda_it }.
\end{align}
Using this in \eqref{eq:psum}, we get
\begin{align}
    \PP_\ybf(\tau_C\geq t)\leq \sum_{i=1}^{\infty} \lambda_{i+1}e^{-\frac{1}{2} \lambda_i\beta_i t}= \sum_{i=1}^{\infty} \lambda_{i+1}e^{-\frac{1}{2} \sqrt{\lambda_i} t}.
\end{align}
Note that 
$$\frac{1}{2}\sqrt{\lambda_i}t\geq \frac{t}{2} + \frac{1}{4}\sqrt{\lambda_i}, \mbox{ for } t\geq 1, \; i \ge 1.$$
Thus, 
$$\PP_\ybf(\tau_C\ge t)\le e^{-\frac{1}{2}t}\sum_{i=1}^\infty \lambda_{i+1}e^{-\frac{1}{4}\sqrt{\lambda_i}} = \Lambda e^{-\frac{1}{2}t}
,\;\;\text{for }t\ge 1,$$
where $\Lambda := \sum_{i=1}^\infty \lambda_{i+1}e^{-\frac{1}{4}\sqrt{\lambda_i}} = \sum_{i=1}^\infty 4^{i+1}e^{-2^{i-2}}< \infty$.
A straightforward argument now shows that, for $0<\eta<\frac{1}{2}$,

\begin{equation}\label{eq:exp.bd}
    \EE_\ybf e^{\eta \tau_C}\le e^\eta + 2\eta\Lambda \frac{1}{1-2\eta} e^{\eta-\frac{1}{2}}.
\end{equation}
The result follows.
\end{proof}

We will now use \cite[Theorem 5.2]{down1995exponential} to complete the proof of uniform ergodicity. For this we show in the next result that the Markov family $\{\PP_\ybf \}_{\ybf\in\RR_+^{n-1}}$ is $\psi$-irreducible  (recall this notion from \cite{down1995exponential}).

 Recall the measure $\nu$ and the constant $\epsilon$ obtained in Lemma \ref{lem:minor.cnd} for which the minorization condition \eqref{minor} holds. 
 \begin{proposition}\label{prop:irr}
    Define the measure $\psi$ on $\clb(\RR_+^{n-1})$ as $\psi(A):= \nu(A)$, $A\in \clb(\RR_+^{n-1}).$ Then the Markov family $\{\PP_\ybf \}_{\ybf\in\RR_+^{n-1}}$ is $\psi$-irreducible. 
 \end{proposition}
 \begin{proof}
     To prove $\psi$-irreducibility it suffices to show that, for each $\ybf\in \RR_+^{n-1}$ and $B\in \clb(\RR_+^{n-1})$ such that $\psi(B)>0$
      $$\EE_\ybf \int_0^\infty \mathbf{1}_{\{\Ybf(t)\in B\}}\;dt>0.$$ 
      Now fix $B \in \clb(\RR_+^{n-1})$ with $\psi(B)>0$.
     Recall the set $C$ defined in \eqref{eq:628n}. Since $C$ is compact, by Lemma \ref{lem:minor.cnd}, there is an $\epsilon>0$ such that
     $$\PP^t(\ybf', B)\ge \eps \nu(B),$$ holds for any $t\in [1,2]$, $\ybf' \in C$.
     
     Fix $t_0>0$.  Then for any $t \in [t_0+1,t_0+2]$ 
     $$\PP^t(\ybf, B)=\int_{\RR_+^{n-1}} \PP^{t-t_0}(\ybf',B)\;\PP^{t_0}(\ybf, d\ybf')\ge \eps \nu(B)\PP^{t_0}(\ybf,C), \quad \ybf \in \RR_+^{n-1}.$$
     Thus
     $$\EE_\ybf \int_0^\infty \mathbf{1}_{\{\Ybf(t)\in B\}}\;dt = \int_0^\infty \PP^t(\ybf,B)\; dt \ge \int_{t_0+1}^{t_0+2}\PP^t(\ybf,B)\; dt
     \ge \eps \nu(B)\PP^{t_0}(\ybf,C), \quad \ybf \in \RR_+^{n-1}.$$
      Thus, to complete the proof, it suffices to show that there exists $t_0>0$, such that for any $\ybf\in \RR_+^{n-1}$, $\PP^{t_0}(\ybf,C)>0$. 

     This can be easily seen using a similar event construction as in Lemma \ref{lem:minor.cnd}. We split the time interval $[0,t_0]$ into $(n-1)$ intervals of equal length. Consider the event where in each of these intervals, at most one particle jumps which jumps exactly once.  Furthermore, these jumps occur sequentially (the possible jump of $X_{i+1}$ is in the $i$-th interval; the leading particle does not jump) and the jump sizes of the particles are such that, following the jump, the corresponding gaps become smaller than $3n^{-1}$. If the $i$-th gap is already smaller than $3n^{-1}$, then no particles move in the $i$-th time interval. Under this event,  $\Ybf(t_0) \in C$. Since, this event has a strictly positive probability (of course, depending on $\ybf\in \RR_+^{n-1}$), the result follows.
 \end{proof}

 We can now complete the proof of uniform ergodicity.
\subsection{Proof of Theorem \ref{thm:unif.erg}}
Recall $\eta$ from Lemma \ref{lem:exp.mom.bd} and define the function $V_0$ as $$V_0(\ybf):= 1-\frac{1}{\eta}+\frac{1}{\eta}\EE_\ybf e^{\eta \tau_C},\;\;\ybf\in \RR_+^{n-1}.$$  Lemma \ref{lem:minor.cnd} implies that, in the terminology of Down, Meyn and Tweedie (cf. \cite[Section 3]{down1995exponential}), the set $C$ is $\eps c(t) \nu$-petite for any $t>0$ for the Markov family $\{\PP_\ybf\}_{\ybf\in\RR_+^{n-1}}$. By Lemma \ref{lem:exp.mom.bd}, $V_0$ is a bounded function on $\RR_+^{n-1}$. This shows that the conditions of \cite[Theorem 6.2]{down1995exponential} are satisfied. Consequently, $V_0$ 
satisfies the drift condition $\cld_T $ in \cite[Section 5]{down1995exponential}. Thus, using Proposition \ref{prop:irr}, and noting that $V_0 \ge 1$, the result is now immediate from \cite[Theorem 5.2]{down1995exponential}.
\hfill
\qedsymbol

\section{Leader with Exponential Jump-Sizes}\label{sec:exp-case}
In this section, we consider the case where 
$\theta$ is Exp$(1)$, namely the leading particle jumps forward by random lengths, which are independent Exp$(1)$-valued random variables. We give an explicit formula for  the stationary distribution in this case. Thereafter, we obtain bounds on the time required for the process to mix in terms of the system size, $n$.
\subsection{Stationary density}
In this section, we prove Theorem \ref{thm:stat.exp.char}. The proof relies on 
verifying that the density of $\pi_n = \mbox{Exp}(1)^{\otimes(n-1)}$ satisfies the stationary adjoint equation for the Markov process, which in view of the uniqueness of stationary distributions established in Theorem \ref{thm:stat.ex} proves the result.
\subsubsection{Proof of Theorem \ref{thm:stat.exp.char}}

\color{black} To show that the law of $\pi_n$ is given by Exp$(1)^{\otimes (n-1)}$, we first calculate the adjoint of the generator $\cll_n$, denoted by $\cll_n^*$. The adjoint $\cll_n^*$ is characterized by the identity
\begin{equation}\label{adj.id}
\int_{\ybf\geq0}f(\ybf) \cll_ng(\ybf)\; d\ybf =\int_{\ybf\geq0}g(\ybf) \cll_n^* f(\ybf)\; d\ybf,
\end{equation}
which holds for all measurable functions $f,\;g: (0,\infty)^{n-1}\to\RR$ with compact support.
After computing $\cll_n^*$ we will show that $$\cll_n^* \Pi_n(\ybf)=0, \; \ybf \in \RR_+^{n-1}$$ where $\Pi_n$ is the density of a random variable with law Exp$(1)^{\otimes (n-1)}$: $$\Pi_n(\ybf)=\exp\left(-\sum_{i=1}^{n-1}y_i \right)\prod_{i=1}^{n-1}\mathbf{1}_{y_i>0}, \quad \ybf\in \RR_+^{n-1}.$$
Now fix $f,g$ as above. Using \eqref{gen.alt}, we have
\begin{multline}\label{adgeneq}
    \int_{\ybf\geq0}f(\ybf) \cll_ng(\ybf)\; d\ybf = \int_{\ybf\geq 0}\int_0^{y_{n-1}} f(\ybf) g(\ybf -ue_{n-1})\;du\;d\ybf \\
    + \sum_{i=1}^{n-2}\int_{\ybf\geq 0}\int_0^{y_i} f(\ybf) g(\ybf+u(e_{i+1}-e_i))\;du\;d\ybf \\
    +\int_{\ybf\geq 0}\int_0^\infty f(\ybf) g(\ybf+ue_1) e^{-u}\;du \;d\ybf -\int_{\ybf\geq 0}f(\ybf) g(\ybf)\left(\sum_{i=1}^{n-1}y_i+1\right)d\ybf.
\end{multline}
For the first term on the RHS of \eqref{adgeneq}, with an interchange of the order of the integration and using the substitution, $\zbf=\ybf-ue_{n-1}$, we have
\begin{align*}
    \int_{\ybf\geq 0}\int_0^{y_{n-1}} f(\ybf) g(\ybf -ue_{n-1})\;du\;d\ybf &= \int_0^\infty \int_{\ybf \geq u e_{n-1}} f(\ybf)g(\ybf -ue_{n-1})\;d\ybf\;du\\
    & = \int_0^\infty \int_{\zbf\geq 0} g(\zbf)f(\zbf+ue_{n-1})\;d\zbf\;du.
    \end{align*}
For each of the quantities under the summation in the second term of the RHS of \eqref{adgeneq}, we can follow similar steps as above with the respective substitution, $\zbf=\ybf+u\cdot(e_{i+1}-e_i)$, to get 
\begin{align*}
\int_{\ybf\geq 0}\int_0^{y_i} f(\ybf) g(\ybf+u(e_{i+1}-e_i))\;du\;d\ybf&=\int_0^\infty\int_{\ybf\geq ue_i} f(\ybf) g(\ybf+u(e_{i+1}-e_i))\;du\;d\ybf \\
&= \int_0^\infty\int_{\zbf\geq ue_{i+1}} g(\zbf)f(\zbf-u(e_{i+1}-e_i)) \;du\;d\zbf \\
&= \int_{\zbf\geq 0}\int_0^{z_{i+1}} g(\zbf)f(\zbf-u(e_{i+1}-e_i)) \;du\;d\zbf.
\end{align*}
Finally, substituting $\zbf = \ybf + ue_1$ in the third term, we get $$\int_{\ybf\geq 0}\int_0^\infty f(\ybf) g(\ybf+ue_1) e^{-u}\;du \;d\ybf = \int_{\zbf\geq 0}\int_0^{z_1} f(\zbf - ue_1) g(\zbf) e^{-u}\;du \;d\zbf.$$
Thus, we can  rewrite the identity in \eqref{adgeneq} as 
\begin{align*}
    \int_{\ybf\geq0}f(\ybf) \cll_ng(\ybf)\; d\ybf &= \int_{\ybf\geq 0} \int_0^\infty  g(\ybf)f(\ybf+ue_{n-1})\;du\;d\ybf\\
    &\quad +  \sum_{i=1}^{n-2} \int_{\ybf\geq 0}\int_0^{y_{i+1}} g(\ybf)f(\ybf-u(e_{i+1}-e_i)) \;du\;d\ybf \\
    &\quad +  \int_{\ybf\geq 0}\int_0^{y_1} f(\ybf - ue_1) g(\ybf) e^{-u}\;du \;d\ybf - \int_{\ybf\geq 0}f(\ybf) g(\ybf)\left(\sum_{i=1}^{n-1}y_i+1\right)\;d\ybf \\
    &=\int_{\ybf\geq 0}g(\ybf)\left[ \int_0^\infty f(\ybf+ue_{n-1})\;du + \sum_{i=2}^{n-1} \int_0^{y_{i}} f(\ybf-u(e_{i}-e_{i-1})) \;du\right. \\
    &\quad + \left.\int_0^{y_1} f(\ybf - ue_1)  e^{-u}\;du - f(\ybf) \left(\sum_{i=1}^{n-1}y_i+1\right)\right]\;d\ybf.
\end{align*}
Thus, comparing with \eqref{adj.id}, we obtain the following formula for the adjoint operator: 
\begin{align}
    \cll_n^* f(\ybf) &= \int_0^\infty f(\ybf+u e_{n-1})\;du\; + \sum_{i=2}^{n-1} \int_0^{y_i} f(\ybf-u(e_{i}-e_{i-1}))\;du\;+\int_0^{y_1}f(\ybf-u e_1) e^{-u}\;du\\
    &\quad\quad-f(\ybf)\left(\sum_{i=1}^{n-1}y_i+1\right).
\end{align}
It is now easily verified that, for $\bfy \in \RR_+^{n-1}$,
\begin{align*}
    \cll_n^* \Pi_n(\ybf) = e^{-\sum_{i=1}^{n-1}y_i}\left[\int_0^\infty e^{-u}\;du + \sum_{i=2}^{n-1}\int_0^{y_i}du+\int_0^{y_1}du-\sum_{i=1}^{n-1}y_i-1\right] =0.
\end{align*}
The result follows.
\hfill
\qedsymbol

\subsection{Mixing time - Lower Bound}
In this section we  give lower bounds on the mixing time, $t_\text{mix}$, of the Markov family $\{\PP_{\bfy}\}_{\bfy \in \RR_+^{n-1}}$ as defined in Section \ref{sec:results}. Since $t_\text{mix}$ is defined as the supremum over all initial distributions in $\clp(\RR_+^{n-1})$, for the purpose of the lower bound it is sufficient to work with a suitable sub-class of $\clp(\RR_+^{n-1})$. 

For this, fix any  $\delta \in (0,1)$, and define for each $n \in \NN$, the class of measures $\mathcal{A}_n=\mathcal{A}_n(\delta) \subset \clp(\RR_+^{n-1})$ as $$\mathcal{A}_n(\delta):= \left\{\mu_n\in\clp (\RR_+^{n-1}): \EE_{\mu_n}\left(\sum_{i=1}^{n-1} Y_i(0)\right) \le (1-\delta)(n-1) \text{ and }\text{Var}_{\mu_n}\left(\sum_{i=1}^{n-1}Y_i(0)\right)\leq n-1\right\},$$ where $\EE_{\mu_n}$ and $\text{Var}_{\mu_n}$ are the expectation and variance corresponding to the probability measure, $\PP_{\mu_n}$. Note that
$\mathcal{A}_n(\delta)$ is nonempty, as the measure with unit mass at
$(1-\delta){\bf 1}_{n-1}$ lies in
$\mathcal{A}_n(\delta)$, where ${\bf 1}_{n-1}$ is the $(n-1)$-dimensional vector of ones.
\subsubsection{Proof of the lower bound}\label{subsec:tmixlow}
    The proof is based on identifying a suitable {\em distinguishing statistic} $\phi$ so that the discrepancy between the expected value of $\phi(\Ybf(t))$ and that of $\phi$ at stationarity (namely $\int \phi d\pi_n$) can be bounded from below appropriately. Specifically, we use the fact from \cite[Proposition 7.12]{levin2017markov} that for probability measures $\mu$, $\nu$ on a Polish space $\cls$, and a real-valued function $f$ on $\cls$ with $\int f^2 d \mu <\infty$, $\int f^2 d \nu <\infty$,
    if we have for some $r>0$, $$|\EE_\mu f- \EE_\nu f|\ge r\left[\frac{\text{Var}_\mu(f)+\text{Var}_\nu(f)}{2}\right]^{\frac{1}{2}},$$ then, it implies that 
    \begin{equation}\label{eq:dist.stat}
        \|\mu-\nu\|_\text{TV}\ge 1-\frac{4}{4+r^2}.
    \end{equation}
    We remark that the proof in \cite{levin2017markov} is written for measures on a finite state space but the argument extends to a general Polish space in a straightforward manner.
    
    We will use \eqref{eq:dist.stat} to obtain a lower bound on $\|\PP^t(\mu_n,\cdot)-\pi_n\|_\text{TV}$ for $\mu_n\in \mathcal{A}_n(\delta)$, for appropriately chosen $\delta \in (0,1)$. This will in turn provide us with a lower bound estimate for the mixing time. 
    
    Define the function $\phi:\RR_+^{n-1}\to \RR$, which will serve as our distinguishing statistic, as $$\phi(\ybf):= \sum_{i=1}^{n-1} y_i.$$     Note that, for the stationary distribution $\pi_n$ obtained in Theorem \ref{thm:stat.ex}, we have 
    \begin{equation}\label{eq:stat.E.var}
        \EE_{\pi_n}(\phi)=(n-1) \quad \text{and}\quad \text{Var}_{\pi_n}(\phi)= (n-1).
    \end{equation}
    Also note that, since the leading particle has Exponential jump sizes, for any $\mu_n \in \mathcal{A}_n(\delta)$ and $t \ge 0$
    $\EE_{\mu_n} (\sum_{i=1}^{n-1} Y_i^2(t))$, and hence $\EE_{\mu_n}(\phi(\Ybf(t)))^2$, is finite.
    We now obtain  a suitable lower bound for $|\EE_{\mu_n}(\phi(\Ybf(t)))-(n-1)|$, for ${\mu_n} \in \mathcal{A}_n(\delta)$. 

Note from \eqref{eq:q1}-\eqref{eq:qi} that, for $\bfy \in \RR_+^{n-1}$,
\begin{align*}
        \cll_n \phi(\bfy)) &= \left(1-\frac{y^2_1}{2}\right) + \sum_{i=2}^{n-1} \left(\frac{y^2_{i-1}-y^2_{i}}{2}\right) 
        = 1 - \frac{1}{2} y^2_{n-1}.
    \end{align*} 
    By an argument similar to the one used in \eqref{eq:150nn} (and using  $\EE_{\mu_n}(\phi(\Ybf(t)))^2<\infty$) we see that
 \begin{align*}\EE_{\mu_n}(\phi(\Ybf(t)))&=\EE_{\mu_n}\phi(\Ybf(0)) + \int_0^t \EE_{\mu_n}\cll_n \phi(\Ybf(s))\; ds\\
 &\le (1-\delta)(n-1) + t - \frac{1}{2}\int_0^t \EE_{\mu_n}  Y^2_{n-1}(s)\;ds,
 \end{align*} 
 where the last line follows on using the above formula for $\cll_n\phi$ and noting that,  for any $\mu_n\in \mathcal{A}_n(\delta)$, $$\EE_{\mu_n}\left(\sum_{i=1}^{n-1}Y_i(0)\right)\le (1-\delta) (n-1).$$  
 Thus we have that
  $$\EE_{\pi_n}(\phi) - \EE_{\mu_n}(\phi(\Ybf(t))\geq \delta (n-1) - t + \frac{1}{2}\int_0^t \EE_{\mu_n}  Y^2_{n-1}(s)\;ds.$$
    Thus, for $t \in (0,\frac{\delta n}{2})$ and $n\ge 2$, we  have $$|\EE_{\pi_n}(\phi) - \EE_{\mu_n}(\phi(\Ybf(t))|\geq \left(\delta (n-1)-t\right).$$
    Next, we consider the variances. From \eqref{eq:stat.E.var}, $\text{Var}_{\pi_n}(\phi) = (n-1).$
\newcommand{\Var}{\text{Var}}
    Now, we consider  $\Var_{\mu_n}(\phi(\Ybf(t))$. 
Observe that
\begin{align*}
\Var_{\mu_n}(\phi(\Ybf(t)) &= \Var_{\mu_n}(X_1(t) - X_n(t)) \le \EE_{\mu_n}[X_1(t) - X_n(t)]^2 
\le \EE_{\mu_n}[X_1(t) - X_n(0)]^2 \\
&\le 2\EE_{\mu_n}[X_1(t) - X_1(0)]^2 + 
2\EE_{\mu_n}[X_1(0) - X_n(0)]^2\\
&\le 2(2t +t^2) + 2(n-1) + 2(1-\delta)^2 (n-1)^2.
\end{align*}
Fix $\veps \in (0,\delta/2)$ and $n\ge 2$. Take $t= n\veps$. Then
\begin{align*}
\Var_{\mu_n}(\phi(\Ybf(n\veps))
\le 2 [2n\veps +n^2\veps^2 + (1-\delta)^2(n-1)^2 + n-1].
\end{align*}
Consequently,
$$\sigma^2 := \frac{1}{2}[\Var_{\mu_n}(\phi(\Ybf(n\veps)) + \text{Var}_{\pi_n}(\phi)]
\le \left[\frac{n-1}{2} + 2n\veps +n^2\veps^2 + (1-\delta)^2(n-1)^2 + n-1\right]$$
and
\begin{align*}
A_n(\veps, \delta) &:= \frac{1}{\sigma^2}
|\EE_{\pi_n}(\phi)-\EE_{\mu_n}(\phi(\Ybf(t)))|^2\\
&\ge \left(\delta (n-1)-n\veps\right)^2 \left[\frac{n-1}{2} + 2n\veps +n^2\veps^2 + (1-\delta)^2(n-1)^2 + n-1\right]^{-1}.
\end{align*}
Note that
\begin{align*}
\limsup_{\delta \to 1} \limsup_{\veps\to 0}
\limsup_{n\to \infty} A_n(\veps, \delta)
= \limsup_{\delta \to 1} \limsup_{\veps\to 0}
\frac{(\delta-\veps)^2}{(\veps^2 +(1-\delta)^2)}
= \limsup_{\delta \to 1} \frac{\delta^2}{(1-\delta)^2} = \infty.
\end{align*}
Thus we can find $\delta_0 \in (0,1)$, $\veps_0 \in (0,\delta_0/2)$ and $n_0 \in \NN$ so that for every $n\ge n_0$ and $\mu_n \in \mathcal{A}_n(\delta_0)$,
$$|\EE_{\pi_n}(\phi)-\EE_{\mu_n}(\phi(\Ybf(n\veps_0)))| > 2 \left(\frac{1}{2}[\Var_{\mu_n}(\phi(\Ybf(n\veps_0)) + \text{Var}_{\pi_n}(\phi)]\right)^{1/2}.$$
From \eqref{eq:dist.stat}, we now have that for $n\ge n_0$ and $\mu_n \in \mathcal{A}_n(\delta_0)$,
\begin{align*}
\|\pi_n-\PP^{n\veps_0}(\mu^n, \cdot)\|_\text{TV}\ge 1-\frac{4}{4+2^2} = 1/2.
\end{align*}
Consequently $t_\text{mix}\ge \eps_0 n$ for all $n \ge n_0$, completing the proof of the lower bound in Theorem \ref{thm:tmix}. 
\hfill
    \qedsymbol

\subsection{Mixing time - Upper bound}
In this section, we will prove the upper bound in Theorem \ref{thm:tmix}, namely establish that the mixing time is at most of the order of $n(\log n)^2.$ Before proceeding, we provide a brief outline for our approach.
\subsubsection{Approach}
As discussed in Section \ref{sec:results}, the key idea in this proof is the construction of a suitable coupling and  using the tail probability of the coupling time to obtain an upper bound for the distance between the law of the gap process at a given time and the stationary law. This, in turn, provides us with an  upper bound for the mixing time, $t_\text{mix}$. 

Recall the coupling given in terms of processes $\bfY, \tilde{\bfY}$, which was constructed in Section \ref{sec:coupling}. Under this coupling, one has that if $Y_i(\cdot)\wedge \tilde Y_i(\cdot)$ is reasonably large - say, $O(1)$, then the $i^{th}$ gaps will coalesce in $O(1)$ time. The main challenge is in handling situations where the gaps are very small. Specifically, it is not clear how to obtain apriori control on mixing time in terms of a lower bound on the gap size, uniformly over all gaps.
To handle this difficulty, we introduce a new particle system which captures the worst-case scenario in terms of possible gap configurations. This system has all the particles stacked at the same position, except for the leading particle - which is $O(1)$ distance ahead of the rest. Further, both the leading and the last particles do not move in this system, and thus  we call the system the \emph{frozen boundaries process}. By establishing control over the time required for the last gap to become $O(1)$ in this system, we are able to obtain a suitable estimate for our model.

The proof proceeds by first constructing the frozen boundaries process. Then, we prove that the time required for the last gap to become larger than $(2e)^{-1}$ is of the order of the square of the system size. We will then rigorously establish that the frozen boundaries process is indeed the worst-case scenario; in the sense that the gap between the first two particles `propagate' to the back faster in the original system than in the frozen boundaries process on a suitable ``good set". This together  with  an appropriate estimate on the probability of the corresponding ``bad set" completes the proof of the upper bound.
\subsubsection{Frozen boundaries process}\label{sec:fbp}

For each $1 < m \le n$  consider an $m$-particle system denoted by $\Zbf_m(\cdot):=(Z_1(\cdot),Z_2(\cdot),\dots,Z_m(\cdot))$. The first (i.e. the leading) particle, $Z_1$, is fixed at $1$ and the last particle, $Z_m$, is fixed at 0. That is, $Z_1(t)=1$ and $Z_m(t)=0$, for every $t\ge 0$.  The rest of the particles start from $0$, namely, $Z_i(0)=0$ for $i\in [m]/\{1\}$. Similarly, as in the original particle system, $Z_i(t)\geq  Z_{i+1}(t)$ holds for every $i\in[m-1]$ and $t\ge 0$. Other than the leading particle $(Z_1)$ and the last particle $(Z_m)$, the particles follow the same dynamics as that of the original system. That is, at a given time instant $t$, for every $i\in[m-1]\setminus \{1\}$, the $i^{th}$ particle $Z_i$ has a jump rate of $Z_{i-1}(t)-Z_i(t)$ and its jump size is given by a U$(0,Z_{i-1}(t)-Z_i(t))$-valued random variable. Note that as long as the $i^{th}$ and the $(i-1)^{th}$ particles occupy the same position ($Z_i(\cdot)=Z_{i-1}(\cdot)$) - which happens in this system with positive probability because of the choice of the initial configuration - the $i^{th}$ particle cannot move. Thus, the particles leave the ``$0$" position in the order of their label. We denote the probability measure on the space where $\bfZ_m$ is defined  by $\PP^*$ and the corresponding expectation by $\EE^*$.

In the following lemma, we give an estimate on the time required for the penultimate particle $(Z_{m-1})$ in this system to move ahead by a length of at least $(2e)^{-1}$. Since, the last particle is fixed at 0, this also corresponds to the time needed for the last gap to become at least as large as $(2e)^{-1}$. 
\begin{lemma}\label{lemma.froz.bd}
     Let $\beta(m) \equiv \beta:= \inf \{t>0: Z_{m-1}(t)\geq \frac{1}{2e}\}.$ Then, there exists a  $c \in (0, \infty)$, such that for all $m>1$ $$\EE^*\beta(m) \leq cm^2.$$ 
    
\end{lemma}
\begin{proof}
    Define a sequence of stopping times, $\delta_1=0$, $\delta_i:= \inf\{t\geq \delta_{i-1}: Z_i(t)\geq (1-\frac{1}{m})^{i-1}\}$. Note that, $(1-\frac{1}{m})^m>\frac{1}{2e}$ holds for every $m\geq 2$. Thus, $\beta\leq_d\delta_{m-1}$.

    We  note that, since the leading particle does not move, $\delta_2$ is distributed as Exp$(1/m)$, consequently $\EE^*\delta_2 = m$.

    Consider now $Z_3$. Let $\clf^{\bfZ}_t := \sigma\{\bfZ_m(s): s \le t\}$. Then
    on the event $\{Z_3(\delta_2) \geq \left(1-\frac{1}{m} \right)^2\}$, 
    $\EE^*((\delta_3-\delta_2) \mid \clf^{\bfZ}_{\delta_2}) =0$ and since $Z_2(t+\delta_2) \ge 1- \frac{1}{m}$ for all $t\ge 0$,
    on the event
    $\{Z_3(\delta_2) < \left(1-\frac{1}{m} \right)^2\}$, 
    $$\cll\left (\delta_3- \delta_2 \mid \clf^{\bfZ}_{\delta_2}\right) \leq_d 
    \mbox{Exp}\left( (1-\frac{1}{m}) - (1-\frac{1}{m})^2\right)$$
    and consequently, on this event,
    $$
    \EE^*\left(\delta_3- \delta_2 \mid \clf^{\bfZ}_{\delta_2}\right) \le \left( (1-\frac{1}{m}) - (1-\frac{1}{m})^2\right)^{-1}
    = \frac{m}{\left(1-\frac{1}{m}\right)}.
    $$
    Combining these observations
    $$\EE^*(\delta_3- \delta_2) \le \frac{m}{\left(1-\frac{1}{m}\right)}.$$


    Continuing similarly, for $Z_i$, if $Z_i(\delta_{i-1})\geq \left(1-\frac{1}{m} \right)^{i-1}$, then $\delta_i=\delta_{i-1}$. If that is not the case,  since $Z_{i-1}(t+\delta_{i-1}) \ge (1- \frac{1}{m})^{i-2}$ for all $t\ge 0$,
    on the event
    $\{Z_i(\delta_{i-1}) < \left(1-\frac{1}{m} \right)^{i-1}\}$, 
    $$\cll\left (\delta_{i}- \delta_{i-1} \mid \clf^{\bfZ}_{\delta_{i-1}}\right) \leq_d 
    \mbox{Exp}\left( (1-\frac{1}{m})^{i-2} - (1-\frac{1}{m})^{i-1}\right)$$
    and consequently, on this event,
    $$
    \EE^*\left(\delta_{i}- \delta_{i-1} \mid \clf^{\bfZ}_{\delta_{i-1}}\right) \le \left( (1-\frac{1}{m})^{i-2} - (1-\frac{1}{m})^{i-1}\right)^{-1}
    = \frac{m}{\left(1-\frac{1}{m}\right)^{i-2}}
    $$
    and thus
    $$\EE^*(\delta_i-\delta_{i-1})\leq \frac{m}{\left(1-\frac{1}{m}\right)^{i-2}}.$$
    Thus, using the fact that $\delta_{m-1}=\sum_{i=2}^{m-1}(\delta_i-\delta_{i-1})$,  we finally have 
    \begin{align}
        \EE^*\beta \leq \EE^* \delta_{m-1}\leq \sum_{i=2}^{m-1} \frac{m}{\left(1-\frac{1}{m}\right)^{i-2}} 
        = m\left[\frac{(\frac{m}{m-1})^{m-2}-1}{\frac{m}{m-1}-1}\right].
    \end{align}
    Since $\lim_{m\to \infty}(\frac{m}{m-1})^{m-2}=e$,  we can a find a constant $c>0$, such that $\left(\frac{m}{m-1}\right)^{m-2}\leq c$ for every $m\in \NN\setminus\{1\}$.
Thus we have that $$\EE^*\beta \leq cm(m-1)\leq cm^2,$$ 
completing the proof of the lemma.
\end{proof}

Having obtained an upper bound for the expected time required for the last gap to become sufficiently large in the frozen boundaries process, we will now prove that this process moves slower than the original particle system driven by any non-decreasing trajectory of the leading particle $X_1(\cdot)$. 

Take any $m \in \mathbb{N}$, $\xbf=(x_1,\dots,x_m) \in \RR^m$, and write $\mathscr{F}_\xbf$ for the class of non-decreasing functions $f:[0,\infty) \rightarrow [x_1,\infty)$. For $\xbf=(x_1,\dots,x_m) \in \RR^m$ and $f \in \mathscr{F}_\xbf$, write $\tilde \PP^{f,m}_\xbf$ for the law of the $m$ particle system $\Xbf_m$ (suppressing $f$ for notational convenience) with $\Xbf_m(0)=\xbf$, $X_1(t)=f(t)$ for all $t \ge 0$, and the remaining particles (given the trajectory of $X_1(\cdot)$) having the same dynamics as in the original system. 


The initial configuration of $\Xbf_m$ is taken to be in the set
\begin{equation}\label{cls0}
    \cls_0(m):= \{(x_1, x_2,\dots,x_{m-1},x_m)\in \RR^m: x_1\ge x_2\ge \dots\ge x_{m-1}\ge x_m=0,\; x_1 \ge 1 \}.
\end{equation}
We will suppress $m$ in $\cls_0(m)$ when clear from the context.
Note that when $\bfX_m$ has initial configuration in $\cls_0$, $Z_i(0)\leq X_i(0)$ for every $i\in [m]$. We now show that with such an initial configuration, for every $i\in[m]$ and $t\ge 0$, and any such choice $f$ of the path of the leading particle, $Z_i(t)\leq_d X_i(t)$.
\newcommand{\sbf}{\textbf{s}}
\begin{lemma}\label{lemma.dom}
    Fix any $m \in \mathbb{N}$. For any $t\ge 0$ and $r\ge 0$,
    $$\PP^*( Z_i(t)\ge r)\leq \inf_{\sbf \in \cls_0}\inf_{f \in \mathscr{F}_\sbf}  \tilde \PP^{f,m}_\sbf (X_i(t) \ge r),$$
    for every $i\in [m].$
\end{lemma}
\begin{proof}
    It suffices to construct, for every $\sbf \in \cls_0$ and $f \in \mathscr{F}_\sbf$, a coupling $(\tilde \Xbf_m, \tilde \Zbf_m)$ with joint law denoted by $\PP'_\sbf$ such that $\tilde \Xbf_m \overset{d}{=} \Xbf_m $ where $\Xbf_m$ has law $\PP^{f,m}_\sbf$, $\tilde \Zbf_m \overset{d}{=} \Zbf_m $, and for any $t\geq0$ and $j\in [m]$, $\tilde Z_j(t) \leq \tilde X_j(t) $, $\PP'_\sbf$-a.s.
    
    We consider the following coupling. We let $\tilde \Xbf_m(0)=\Xbf_m(0) = \sbf$ and $\tilde\Zbf_m(0)=\Zbf_m(0) = (0, \ldots , 0, 1)$. Writing $\sbf = (x_1, \ldots , x_{m-1}, 0)$, 
    and with $\tilde X_1(t)=f(t),\, t \ge 0$, we construct the remaining processes as follows.
    At any time instant $t$, a (possible) \emph{jump event} occurs at rate $\tilde X_1(t)$, described as follows. First, we draw a random variable, $U\sim$U$(0,\tilde X_1(t))$.
    Jumps in the $\tilde \Xbf_{m|1} = (\tilde X_2, \ldots, \tilde X_m)$ process are described as below.
    \begin{enumerate}
        \item If $U\notin (\tilde X_m(t), \tilde X_1(t))$, no particle in the $\tilde \Xbf_{m|1}$ process jumps.
        \item If $U\in (\tilde X_m(t), \tilde X_1(t))$, then, we find $i\in [m]\setminus\{1\}$ such that $U\in (\tilde X_i(t), \tilde X_{i-1}(t))$. The $i^{th}$ particle, $\tilde X_i$, takes a jump of size, $U-\tilde X_i(t)$ and remaining particles stay unchanged.
    \end{enumerate}
    At the instances $t$ of possible jump events described above, possible jumps of  $\tilde \Zbf_{m|1,m}= (\tilde Z_2, \ldots , \tilde Z_{m-1})$ occur 
    as follows. With $U$ as above,
    \begin{enumerate}
        \item If $U\geq 1$ or $U \le \tilde Z_{m-1}(t)$, no particle in the $\tilde \Zbf_{m|1,m}$ process jumps.
        \item If $\tilde Z_{m-1}(t)<U<1$,  we find $j\in [m]\setminus\{1,m\}$ such that $U\in (\tilde Z_j(t), \tilde Z_{j-1}(t))$. The $j^{th}$ particle, $\tilde Z_j$, takes a jump of size, $U-\tilde Z_j(t)$ and remaining particles stay unchanged.
    \end{enumerate}
    To verify that this is indeed a valid coupling, observe that at time $t$, the jump rate of the $i^{th}$ particle in the $\tilde \Xbf_{m|1}$ process is given by $$\frac{\tilde X_{i-1}(t)-\tilde X_i(t)}{\tilde X_{1}(t)-\tilde X_m(t)}\cdot\frac{\tilde X_{1}(t)-\tilde X_m(t)}{\tilde X_{1}(t)}\tilde X_{1}(t)=\tilde X_{i-1}(t)-\tilde X_i(t).$$ Further, conditional upon $U\in (\tilde X_i(t),\tilde X_{i-1}(t))$, the law of $U$ is given by U$(\tilde X_i(t),\tilde X_{i-1}(t))$. Hence, the jump size distribution of $i^{th}$ particle in $\tilde \Xbf_m$ matches that in the $\Xbf_m$ process. The verification for the $\tilde \Zbf_m$ process follows similarly. This can be checked formally by computing the generators of the Markov processes $\tilde \Xbf_m$ and $\tilde \Zbf_m$.
    
    Next, we argue that, for every $j\in [m]$, $\tilde Z_j(t)\leq \tilde X_j(t)$, for every $t\geq 0$. Note that, since $\tilde Z_1$ is fixed at 1, we  only consider the non-trivial case: $j\neq1$. Also recall that $\tilde Z_j(0)\leq \tilde X_j(0)$ holds for every $j$. Hence, it is enough to show, that if a jump event, as described above, occurs at time $t$, as a result of which, for some $j \in [m-1]$, $\tilde Z_j$ jumps, then, $\tilde Z_j(t)\leq \tilde X_j(t)$, if $\tilde Z_\ell (t-) \leq \tilde X_\ell (t-)$ for all $\ell \in [m-1]$. 
    
    For this purpose, let us recall the uniform random variable $U$ used earlier to describe the jump dynamics. Since we are considering the case where $\tilde Z_j$ jumps at time $t$, we must have that $U$ lies ahead of $\tilde Z_j$, but behind $\tilde Z_{j-1},$ that is $U\in (\tilde Z_j(t-), \tilde Z_{j-1}(t-))$. Furthermore, after this jump, the new position of $\tilde Z_j$ is given by $U$, that is $\tilde Z_j(t)=U.$ Now, let us first consider the case where $U$ lies ahead of $\tilde X_j$, that is $U>\tilde X_j(t-)$. Since we already know that $U$ lies behind $\tilde Z_{j-1}$ ($U<\tilde Z_{j-1}(t-)$), this implies that $U$ necessarily lies behind $\tilde X_{j-1}$ ($U<\tilde X_{j-1}(t-)$). Hence, we have in this case, $U\in (\tilde X_j(t-), \tilde X_{j-1}(t-))$ which implies that at the given time instant $t$, both $\tilde X_{j}$ and $\tilde Z_{j}$ jump to the same position on the real line, given by $U$. That is, we get $\tilde X_{j}(t)=\tilde Z_{j}(t)=U$. Hence, the inequality between the processes still holds after the jump, i.e. $\tilde Z_\ell (t) \leq \tilde X_\ell (t)$ for all $\ell \in [n-1]$.  On the other hand, if $U \le \tilde X_{j}(t-)$, then the inequality clearly holds as following the jump, we have $\tilde Z_{j}(t)=U \le \tilde X_j(t)$. Thus we conclude
that for every $j\in [m]$ and for any $t\geq 0$,  $\tilde Z_j(t) \leq \tilde X_j(t) $, $\PP'_\sbf$-a.s. and the result follows.
\end{proof}

\begin{remark}\label{rmk:dom}
   Let $n \ge 3$ and $3 \le m \le n$. As an immediate consequence of Lemma \ref{lemma.dom}, we now have the following stochastic dominance property for the $n'$-th particle in the $n$-dimensional system for any $m-1 \le n' \le n-1$, when the $(n'+1)$-th particle starts at $0$ and the $(n'-m+2)$-th particle starts from above $1$.
   Let,
   $$\cls_0^n(m,n'): = \{(x_1, x_2,\dots, x_n)\in \RR^n: x_1\ge x_2\ge \dots\ge x_{n-1}\ge x_n, \; x_{n'+1}=0,\; x_{n'-m+2} \ge 1 \}.$$
   Then, for any $m-1 \le n' \le n-1$,
    $$\PP^*\left( Z_{m-1}(t) \ge (2e)^{-1}\right)\leq \inf_{\sbf \in \cls_0^n(m,n')}\tilde  \PP_\sbf \left(X_{n'}(t) \ge (2e)^{-1}\right), \quad t \ge 0.$$ This 
    says that, letting
    $\beta^{n'}_X:= \inf\{t\ge 0: X_{n'}(t)\ge (2e)^{-1} \}$, we have, 
    $$\sup_{m-1 \le n' \le n-1} \ \sup_{\sbf\in \cls_0^n(m,n')}\tilde \EE_\sbf\beta^{n'}_X\le \EE^*\beta(m).$$ 
\end{remark}
As an immediate consequence of Remark \ref{rmk:dom} and Lemma \ref{lemma.froz.bd}, we have the following corollary.
\begin{corollary}\label{cor:xi}
   Fix $n \ge 3$. For $3 \le m \le n$ and $m-1 \le n' \le n-1$, let $$\xi^{n'}(m):= \inf\{t\ge0:X_{n'+1}(t)>0\}\wedge \beta^{n'}_X.$$ Then, there exists $c>0$, such that for every $m= 3, \ldots, n$,
    $$\sup_{m-1 \le n' \le n-1} \ \sup_{\sbf\in \cls_0^n(m,n')}\tilde \EE_\sbf \xi^{n'}(m) \le cm^2.$$
\end{corollary}

We will now use the above corollary to complete
the proof of the mixing time upper bound.

\subsubsection{Proof of the upper bound}\label{subsec:tmixup}
    As discussed earlier, the proof is based on considering a suitable coupling and obtaining bounds on $t_\text{mix}$ using the tail probability of the coupling time by appealing to the inequality in \eqref{eq:tmix.up}. Thus, once again, we consider the coupling between particles $(\Xbf, \tilde \Xbf)$ (resp. gaps ($\Ybf,\tilde \Ybf$)) introduced in Section \ref{sec:coupling}, where $\Ybf(0)$ has some initial distribution $\mu \in \clp(\RR_+^{n-1})$  and $\tilde \Ybf$ is stationary: for every $t\ge0$, $\tilde \Ybf(t)\sim\pi_n$. 
    Denote the probability measure on the space where these processes are defined as $\PP_{\mu, \pi_n}$.
    We denote the corresponding coupling time by $\tau_\text{coup}$. The idea of the proof is to identify a `good' set on which we can invoke Corollary \ref{cor:xi}. We will then use the stationarity of $\tilde \Ybf$ to control the probability of the complement of this good set. Finally we combine these two steps to obtain an upper bound for the tail probability of $\tau_\text{coup}$.
    
    We first construct the good set. Fix $r>0$ and suppose that  $n$ is large enough so that $n-1 \ge \lfloor r\log n\rfloor + 1$. Consider the set: $$A_{n,r}:= \{\exists  i\in [\lfloor r\log n\rfloor + 1, n-1] \text{ and } l\in [n^3],\;  \text{such that } \tilde Y_j(l)<1 ,\; \forall j \in [i-\lfloor r\log n\rfloor\, , \,i] \}.$$ 
    Recall that $\tilde \Ybf$ is stationary and therefore the probability of $A_{n,r}$ can be estimated as follows. 
    \begin{align}
        \PP_{\mu, \pi_n}(A_{n,r})=&\PP_{\mu, \pi_n}\left(\bigcup_{l=1}^{n^3}\bigcup_{i=\lfloor r\log n\rfloor + 1}^{n-1} \{ \tilde Y_j(l)<1 ,\; \forall j \in [i-\lfloor r\log n\rfloor\, , \,i]\right) \nonumber\\
        &\leq \sum_{l=1}^{n^3}\sum_{i=\lfloor r\log n\rfloor + 1}^{n-1} \PP_{\mu, \pi_n} (\tilde Y_j(l)<1 ,\; \forall j \in [i-\lfloor r\log n\rfloor\, , \,i])\nonumber\\
        &\le n^3(n- r\log n) \pi_n(\ybf \in \RR_+^{n-1}: y_j < 1,\; \forall j \in [i-\lfloor r\log n\rfloor\, , \,i])\nonumber\\ 
        &\leq n^3(n-r\log n)(1-e^{-1})^{r\log n}.
        \label{eq:bad.set}
    \end{align}
    Observe that on $A_{n,r}^c$ (which is our good set), if we consider any particle in the $\tilde \Xbf$ process, then at all integer times $ \le n^3$, we can find at least one gap among the $r\log n$ gaps preceding it which is larger than 1. This makes the particle configuration locally resemble those described by $\cls_0$ defined in \eqref{cls0} and hence, enables us to compare the system locally to the frozen boundaries process. 

    Recall from Section \ref{sec:coupling} that under our coupling construction, particles (and therefore the gaps) in the two system coalesce sequentially from front to back. Also recall that under our coupling $X_1(t) = \tilde X_1(t)$ for all $t \ge 0$.
    Thus, we define the stopping times: $\tau'_0=0$ and for $k\in [n-1]$,
    $$\tau'_{2k-1}:= \left\{t\geq \lceil\tau'_{2k-2}\rceil: Y_{k}(t)\wedge \tilde Y_{k}(t) \geq \frac{1}{2e} \right\} \wedge \inf \{ t\geq \lceil\tau'_{2k-2}\rceil: Y_{k}(t)=\tilde Y_{k}(t) \}, $$
    $$\text{and }\tau'_{2k}:= \inf \{ t\geq \tau'_{2k-1}: Y_{k}(t)=\tilde Y_{k}(t) \}.$$
    Note that $\tau_\text{coup}\leq \tau'_{2(n-1)}$. 
    
    Recall that, under our coupling, at any given time $t$, the $k^{th}$ gaps will coalesce at rate $Y_{k}(t)\wedge \tilde Y_{k}(t)$, and before the slower particle moves, $Y_{k}(t)\wedge \tilde Y_{k}(t)$ is non-decreasing. Thus, for $t\ge \tau'_{2k-1}$, the coalescence of the $k^{th}$ gap occurs at a rate of at least $(2e)^{-1}$. Thus, for every $k\in [n-1]$, we have 
    \begin{equation}\label{eq:tau_2k}
        \EE_{\mu, \pi_n}(\tau'_{2k} - \tau'_{2k-1})\leq 2e.
    \end{equation}
    Now consider $(\tau'_{2k-1} - \tau'_{2k-2})$.
    We start by considering $k\in [\lfloor r\log n \rfloor]$. Let  $$\tau_{0,k}:= \inf\{t\ge \tau'_{2k-2}: Y_1(t)\wedge \tilde Y_1(t)\ge 1\}.$$ 
    Observe that
    $$\tau_{0,k} \le \tau'_{0,k}:= \inf\{t\ge \tau'_{2k-2}: X_1(t)-X_1(t-) = \tilde X_1(t)-\tilde X_1(t-)>1 \}.$$ Since the probability that a jump of $X_1$ (equivalently, of $\tilde X_1$) will have size larger than 1 is given by $e^{-1}$, we have that the number of jumps of size larger than 1 is a Poisson process with rate $e^{-1}$. Thus, 
    $$\EE_{\mu, \pi_n}(\tau_{0,k} - \tau'_{2k-2}) \leq \EE_{\mu, \pi_n}(\tau_{0,k}' - \tau'_{2k-2})=e.$$
    Furthermore, $(\tau'_{2k-1}-\tau_{0,k})_+$ - that is, the additional time after $\tau_{0,k}$ required for $Y_k(\cdot)\wedge \tilde Y_k(\cdot)$ to become larger than $(2e)^{-1}$ can be estimated using the stopping time $\xi$ defined in Corollary \ref{cor:xi}.
    Indeed, letting $\clf_t \doteq \sigma\{(\Xbf(s), \tilde\Xbf(s)): 0 \le s \le t\}$, conditioned on $\clf_{\tau_{0,k}}$, if we consider the position of the slower particle in the pair $(X_{k+1},\tilde X_{k+1})$ at time $\tau_{0,k}$ along with the locations of the first $k$ particles which have already coalesced, then the configuration of these  $(k+1)$ particles, re-centered to have the $(k+1)$-th slower particle at time $\tau_{0,k}$  at location $0$, belongs to the class $\cls_0$ for $m=k+1.$ Consider the evolution of this $k+1$-sized particle system for $t \ge \tau_{0,k}$, conditioned on $\clf_{\tau_{0,k}}$.  
    Note that, if this slower particle jumps before its gap from the particle in front (namely $Y_k(\cdot)\wedge \tilde Y_k(\cdot)$) becomes larger than $(2e)^{-1}$, then the $(k+1)^{th}$ particles coalesce anyway. 
    Hence, using Corollary \ref{cor:xi} with $n'=k$ and $m= k+1$, we  conclude that for any $k\in[\lfloor r\log n\rfloor]$,  $$\EE_{\mu, \pi_n}(\tau'_{2k-1}-\tau_{0,k})\le c(r\log n +1)^2.$$
    Thus, for $k\in [\lfloor r\log n\rfloor]$, 
    \begin{equation}\label{eq:tau_k.small}
        \EE_{\mu, \pi_n}((\tau'_{2k-1} - \tau'_{2k-2}))\leq e+ c(r \log n +1)^2.
    \end{equation}
    
    Next, we consider $k>r\log n$. 
    Recall that at $\tau'_{2k-2}$, the first $(k-1)$ gaps have already coalesced and so, for every $j\in [k-1]$ and $t>\tau'_{2k-2}$, $Y_j(t)=\tilde Y_j(t)$. 
Consider the set
    $$B_{k,r} = \{\mbox{there exists } j\in \{ k, k-1,\; k-2,\dots,\; k-\lfloor r\log n \rfloor +1\} \mbox{ such that } Y_j(\lceil\tau'_{2k-2}\rceil)=\tilde Y_j(\lceil\tau'_{2k-2}\rceil)\ge 1\}.
    $$
    Then, 
 on the event $B_{k,r}$, the locations of the particles labeled $\{k,k-1,\dots,j\}$ along with the slower particle in the pair $(X_{k+1},\tilde X_{k+1})$ at the time instant $\lceil\tau'_{2k-2}\rceil$, re-centered so that this slower particle is at $0$,   define a configuration  in the class $\cls_0$ with $m=k-j +2$. Thus, using the fact that $k-j\le r\log n-1$, we have from Corollary \ref{cor:xi}, applied with $n'=k$ and $m=k-j+2$, that on $B_{k,r}$,
    \begin{equation}\label{eq:tau_k.large0}
        \EE_{\mu, \pi_n}\left((\tau'_{2k-1} - \tau'_{2k-2}) \, \vert \, \mathcal{F}_{\lceil\tau'_{2k-2}\rceil}\right)\le c(r\log n +1 )^2.
    \end{equation}
    Thus,
     \begin{equation}\label{eq:tau_k.large}
        \EE_{\mu, \pi_n}\left((\tau'_{2k-1} - \tau'_{2k-2}){\bf1}_{B_{k,r}}\right)\le c(r\log n +1 )^2.
    \end{equation}
    We also note that for $k = \lfloor r \log n\rfloor+1, \ldots , n-1$,
    \begin{equation}\label{eq:535n}
    A_{n,r}^c \cap \{\tau'_{2k-2} \le (n-1)n^2\} \subset B_{k,r}.\end{equation}
    Combining this with \eqref{eq:tau_2k} and Markov's inequality, we obtain, for $k = \lfloor r \log n\rfloor+1, \ldots , n-1$,
\begin{align}
    \PP_{\mu, \pi_n}\left(\tau'_{2k} > kn^2, \, \tau'_{2k-2} \le (k-1)n^2, \, A_{n,r}^c\right) &\le \PP_{\mu, \pi_n}\left(\tau'_{2k-1} - \tau'_{2k-2} > n^2/2,  \, B_{k,r}\right)\nonumber\\
    &\quad + \PP_{\mu, \pi_n}\left(\tau'_{2k} - \tau'_{2k-1} > n^2/2\right)\nonumber\\
    &\le \frac{2c(r\log n +1 )^2}{n^2} + \frac{4e}{n^2}.\label{eq:459n}
\end{align}
Using \eqref{eq:bad.set},
\begin{align*}
 \PP_{\mu, \pi_n}\left(\tau'_{2(n-1)} > (n-1)n^2\right) &\le  
\PP_{\mu, \pi_n}\left(\tau'_{2(n-1)} > (n-1)n^2, A_{n,r}^c\right) + \PP_{\mu, \pi_n}(A_{n,r})\\
&\le  \PP_{\mu, \pi_n}\left(\tau'_{2(n-1)} > (n-1)n^2, A_{n,r}^c\right) + n^3(n-r\log n)(1-e^{-1})^{r\log n}.
\end{align*}
Also, using \eqref{eq:459n},
\begin{align*}
\PP_{\mu, \pi_n}\left(\tau'_{2(n-1)} > (n-1)n^2, A_{n,r}^c\right) &\le 
\PP_{\mu, \pi_n}\left(\tau'_{2(n-1)} > (n-1)n^2, \tau'_{2(n-2)} \le  (n-2)n^2, A_{n,r}^c\right)
\\
&\quad + \PP_{\mu, \pi_n}\left(\tau'_{2(n-2)} >  (n-2)n^2, A_{n,r}^c\right)\\
&\le \frac{2c(r\log n +1 )^2}{n^2} + \frac{4e}{n^2}
+ \PP_{\mu, \pi_n}\left(\tau'_{2(n-2)} >  (n-2)n^2, A_{n,r}^c\right).
\end{align*}
Proceeding similarly, using \eqref{eq:459n}, with $k= n-2, n-3, \ldots , \lfloor r \log n\rfloor +1$, and using \eqref{eq:tau_2k} and \eqref{eq:tau_k.small} for $k\in [\lfloor r\log n \rfloor]$, we now get for some $c'>0$ (not depending on $n$)
\begin{align}\label{eq:largetail}
    \PP_{\mu, \pi_n}\left(\tau'_{2(n-1)} > (n-1)n^2\right) 
    &\le \frac{c'(r\log n +1 )^2}{n} + n^3(n-r\log n)(1-e^{-1})^{r\log n},
\end{align}
We now estimate 
$\PP_{\mu, \pi_n}\left(\tau'_{2(n-1)} > t\right)$ for $t \in (0, (n-1)n^2)$. From \eqref{eq:tau_k.large} and \eqref{eq:535n} we see that
\begin{multline}
    \EE_{\mu, \pi_n}\left((\tau'_{2k-1} \wedge (n-1)n^2 - \tau'_{2k-2} \wedge (n-1)n^2) \, \textbf{1}_{A_{n,r}^c}\right)\\
    \le \EE_{\mu, \pi_n}\left((\tau'_{2k-1}  - \tau'_{2k-2})  \, \textbf{1}_{\{A_{n,r}^c, \tau'_{2k-2} \le (n-1)n^2\}}\right)
    \le \EE_{\mu, \pi_n}\left((\tau'_{2k-1}  - \tau'_{2k-2})  \, \textbf{1}_{B_{k,r}}\right)
    \le c(r\log n +1 )^2. \label{eq:tinter}
\end{multline}
Combining \eqref{eq:tinter} with \eqref{eq:tau_2k} and summing over $k$, we get
\begin{align}\label{eq:tinterbd}
    \EE_{\mu, \pi_n}\left((\tau'_{2(n-1)} \wedge (n-1)n^2) \textbf{1}_{A_{n,r}^c}\right) \le 2e n + cn (r\log n +1 )^2.
\end{align}
Finally, for $t \in (0, (n-1)n^2)$, using \eqref{eq:bad.set}, \eqref{eq:largetail} and \eqref{eq:tinterbd}, we obtain
\begin{align*}
    \PP_{\mu, \pi_n}(\tau_\text{coup}>t) &\le \PP_{\mu, \pi_n}\left(\tau'_{2(n-1)} > t\right)\\
    &\le \PP_{\mu, \pi_n}\left(\tau'_{2(n-1)} > (n-1)n^2\right) + \PP_{\mu, \pi_n}\left(\tau'_{2(n-1)} \wedge (n-1)n^2 > t, \,A_{n,r}^c \right) + \PP_{\mu, \pi_n}(A_{n,r})\\
    & \le \frac{c'(r\log n +1 )^2}{n} + \frac{2e n + cn (r\log n +1 )^2}{t} + 2 n^3(n-r\log n)(1-e^{-1})^{r\log n}.
\end{align*}
Now choosing $n_0$ and $r$ sufficiently large and taking 
$t= c^* n(\log n)^2$ for a sufficiently large $c^*$, we have for all $n \ge n_0$,
$$
\PP_{\mu, \pi_n}(\tau_\text{coup}> c^* n(\log n)^2) \le \frac{1}{4},
$$
which implies $t_\text{mix}\le c^* n(\log n)^2$ for $n \ge n_0$. For $n < n_0$, the finiteness of $t_\text{mix}$ follows from Theorem \ref{thm:unif.erg} and thus, we can choose $c^*$ large enough so that the above bound holds for all $n$. Since the choice of $c^*, r$ and $n_0$ is independent of $\mu$, this completes the proof of the upper bound.
    \hfill
\qedsymbol

\subsection{Proof of Theorem \ref{thm:tmix}}
Section \ref{subsec:tmixlow} proves the lower bound in Theorem \ref{thm:tmix} while Section \ref{subsec:tmixup} proves the upper bound. This completes the proof of the theorem. 
\hfill
\qedsymbol
\subsection{Functional Limit Theorem}
In this section we prove Corollary \ref{cor:fclt}.
Recall that for a Polish space $\cls$, $\clp(\cls)$ denotes the space of probability measures on $\cls$ equipped with the topology of weak convergence.
This topology can be metrized by the {\em bounded-Lipschitz} distance defined as follows. Let
$$\mbox{BL}(\cls) := \{f: \cls \to \RR \mbox{ such that } \sup_{x\in \cls} |f(x)| \le 1, \; |f(x)-f(y)| \le |x-y| \mbox{ for all } x,y\in \cls\}.$$
For $\mu,\nu \in \clp(\cls)$ define
$$\dbl(\mu, \nu) := \sup_{f \in \mbox{BL}(\cls)}\left |\int f d\mu - \int f d\nu\right|.$$
Then $\dbl$ defines a distance on $\clp(\cls)$ and a sequence $\mu_n$ in $\clp(\cls)$ converges weakly to $\mu \in \clp(\cls)$ if and only if $\dbl(\mu_n, \mu) \to 0$.

For $n \in \NN$, consider on some probability space random variables
$(X^n_1(\infty), \ldots X^n_n(\infty))$ such that
$X^n_1(\infty)=0$ and
with $Y^n_i(\infty) = X^{n}_i(\infty) - X^{n}_{i+1}(\infty)$, $i = 1, \ldots , n-1$,
$Y^n_1(\infty), \ldots , Y^n_{n-1}(\infty)$ are iid Exp$(1)$. 
Define $\hat X^n_i(\infty) := X^n_i(\infty)/n$,
$\hat Y^n_j(\infty) := Y^n_j(\infty)/n$, $i \in [n]$, $j \in [n-1]$.
For $x \in [0,1]$, define
$$
U^n_{\infty}(x) :=  \sqrt{n}\left(\hat X^n_1(\infty) - \hat X^n_{\lfloor nx\rfloor}(\infty) - x\right).
$$
Then, denoting the probability law of $U^n_{\infty}$
on $\cld([0,1]:\RR)$ as $\Theta_n$ and the probability law of a standard Brownian motion $\{W(t):0\le t \le 1\}$ on $\cld([0,1]:\RR)$ as $\Theta$, we have by Donsker's theorem that 
$\dbl(\Theta_n, \Theta) \to 0$ as $n\to \infty$.
Now let $U^n$ be as in the statement of the corollary and denote by $\bar \Theta^n$ the distribution of $U^n$ on $\cld([0,1]:\RR)$.
Then
$$
\dbl(\bar \Theta^n, \Theta) \le \dbl(\bar \Theta^n, \Theta^n) + \dbl(\Theta^n, \Theta).
$$
Also, with $d$ as in \eqref{eq:913n} and $c_2$ as in Theorem \ref{thm:tmix}, we have from this theorem and \eqref{eq:911n}, 
\begin{align*}
\dbl(\bar \Theta^n, \Theta^n) &\le \|\bar \Theta^n-\Theta^n\|_{\text{TV}} \le \|\PP^{n^{-1}t_n}(\mu_n, \cdot) - \pi_n\|_{\text{TV}}\\
&\le d\left(\frac{\alpha_n}{c_2} c_2 n(\log n)^2\right)
\le d\left(\frac{\alpha_n}{c_2} t_{\text{mix}}\right) \le 2^{-\alpha_n/c_2}.
\end{align*}
Thus
$$\dbl(\bar \Theta^n, \Theta) \le 2^{-\alpha_n/c_2} + \dbl(\Theta^n, \Theta).$$
The result follows on sending $n\to \infty$.

\section{Leader jumps - Power law}\label{sec:heavy-tail}
In this section, we consider the case where the leader's jump sizes have a heavy-tailed distribution. We prove that if the leader's jump sizes have $k$ finite moments, then, at stationarity, the gap sizes will have at least $(k+1)$ finite moments.
\subsection{Proof of Theorem \ref{thm:poly_tails}}
    We begin with an outline of the proof. Using the fact, that $\EE_\theta Z^k<\infty$, we will first show that the $(k+1)^{th}$ moment of first gap, $Y_1$, is finite under stationarity. We will then show that the existence of the $(k+1)^{th}$ moment for the $(i-1)^{th}$ gap at stationarity implies the same for the $i^{th}$ gap. Using this argument recursively will complete the proof of the theorem. Both of these steps are based on using  drift properties of  suitable Lyapunov functions.
   
    Define $\psi_1:\RR_+^{n-1}\to\RR$  as $\psi_1(\ybf)=y_1^k$. Then
    $$\cll_n\psi_1(\ybf) = \EE_\theta[(y_1+Z)^k-y_1^k] + y_1\EE_U [U^k y_1^k -y_1^k].$$
   Writing $\mu_j$ to denote the $j^{th}$ moment of $Z$, and noting that the $k^{th}$ moment of a U$(0,1)$-valued random variable is given by $\frac{1}{k+1}$, we obtain 
    \begin{align}
        \cll_n\psi_1(\ybf) &=  \sum_{j=1}^k \binom{k}{j}y_1^{k-j}\mu_j - \frac{k}{k+1}y_1^{k+1}.
    \end{align}
    Since the dominant term on the right hand side above is $-\frac{k}{k+1}y_1^{k+1}$ when $y_1$ is large, we can  find a finite constant $c_k$ such that 
    \begin{equation}\label{eq:g1}
        \cll_n \psi_1(\ybf) \leq -\frac{k}{2(k+1)}y_1^{k+1} + c_k \mbox{ for all } \ybf \in \RR_+^{n-1}.
    \end{equation}
Fix $m \in \NN$ and define $\psi_1^{(m)}: \RR_+^{n-1} \to \RR$ as $\psi_1^{(m)}(\ybf) := y_1^k \wedge m$.
For $t>0$, let $G_t, G_t^{(m)}: \RR_+^{n-1} \to \RR$ be defined as
\begin{equation}\label{eq:ggm}
G_t(\ybf) := \frac{1}{t}[\psi_1(\ybf) - \EE_{\ybf}(\psi_1(\Ybf(t)))], \;
G_t^{(m)}(\ybf) := \frac{1}{t}[\psi^{(m)}_1(\ybf) - \EE_{\ybf}(\psi^{(m)}_1(\Ybf(t)))].
\end{equation}
Note that, for each $t>0$ and $\ybf \in \RR_+^{n-1}$, $|G_t(\ybf)| < \infty$. This is a consequence of the inequality $Y_1(t) \le X_1(t)$ (assuming that the leading particle starts from $0$ without loss of generality) and the finiteness of $\EE_{\theta}Z^k$.
Also note that when $\psi_1(\ybf) \le m$, $G_t^{(m)}(\ybf) \ge G_t(\ybf)$ and when $\psi_1(\ybf) > m$,
$G_t^{(m)}(\ybf) \ge 0$. Thus we have
\begin{equation}
G_t^{(m)}(\ybf) \ge G_t(\ybf) \wedge 0 \; \mbox{ for all } \ybf \in \RR_+^{n-1}.
\end{equation}
Also, letting, for $r>0$, 
$\tau_r := \inf\{t \ge 0: y_1 \ge r\}$, we have
$$
\psi_1(\bfY(t\wedge \tau_r)) - \psi_1(\bfy)
= \int_0^{t\wedge \tau_r} \cll_n\psi_1(\bfY(s)) ds
+ M(t \wedge \tau_r)
$$
where $t \mapsto M(t \wedge \tau_r)$ is a martingale under $\PP_{\ybf}$. Thus, from \eqref{eq:g1},
\begin{equation}\label{eq:1017nn}
\EE_{\ybf}[\psi_1(\bfY(t\wedge \tau_r))]
- \psi_1(\bfy) = \EE_{\ybf} \int_0^{t\wedge \tau_r}
\cll_n\psi_1(\bfY(s)) ds \le c_k t
\end{equation}
Sending $r \to \infty$ and using monotone convergence
\begin{equation}\label{eq:1228n}\EE_{\ybf}[\psi_1(\bfY(t))]
- \psi_1(\bfy) \le c_k t \mbox{ for all } \ybf \in \RR_+^{n-1}.\end{equation}
Thus,
$$G_t^{(m)}(\ybf) \ge G_t(\ybf) \wedge 0
= \left(\frac{1}{t}[\psi_1(\ybf) - \EE_{\ybf}(\psi_1(\Ybf(t)))]\right) \wedge 0 \ge -c_k.$$
Also, using monotone convergence again, for all
$\ybf \in \RR_+^{n-1}$, $G_t^{(m)}(\ybf) \to G_t(\ybf)$ as $m \to \infty$. Using Fatou's lemma and recalling that $\pi_n$ is stationary, we have,
\begin{equation}\label{eq:1030n} \int G_t(\ybf) \pi_n(d\ybf) \le \liminf_{m\to \infty} \int G_t^{(m)}(\ybf) \pi_n(dy) =0.
\end{equation}
From \eqref{eq:g1} and the first equality in \eqref{eq:1017nn},
$$
\EE_{\ybf}[\psi_1(\bfY(t\wedge \tau_r))]
- \psi_1(\bfy) \le -\frac{k}{2(k+1)}
\EE_{\ybf} \int_0^{t\wedge \tau_r} [Y_1(s)]^{k+1} ds + c_kt.
$$
Thus
$$
\frac{1}{t}\EE_{\ybf} \int_0^{t\wedge \tau_r}
 [Y_1(s)]^{k+1} ds \le
 -\frac{2(k+1)}{k}\frac{1}{t}\left(\EE_{\ybf}[\psi_1(\bfY(t\wedge \tau_r))]
- \psi_1(\bfy)\right) + 2c_k\frac{k+1}{k}.
$$
Letting $r\to \infty$ and using monotone convergence (both on the left and the right side),
$$
\frac{1}{t}\EE_{\ybf} \int_0^{t}
 [Y_1(s)]^{k+1} ds \le
 \frac{2(k+1)}{k}G_t(y) + 2c_k\frac{k+1}{k}.
$$
Integrating with respect to $\pi_n$, we now have from \eqref{eq:1030n},
$$
\frac{1}{t}\int_0^{t}
 \EE_{\pi_n} [Y_1(s)]^{k+1} ds \le
 \frac{2(k+1)}{k}\int G_t(y) \pi_n(d\ybf)+ 2c_k\frac{k+1}{k} \le 2c_k\frac{k+1}{k} .
$$
Thus, since $\pi_n$ is stationary
$$
\EE_{\pi_n} [Y_1(t)]^{k+1} \le 2c_k\frac{k+1}{k} <\infty \mbox{ for all } t \ge 0.$$

    Next, let us assume that for some $i= 2, \ldots , n-2$ the $(k+1)^{th}$ moment of the $(i-1)^{th}$ gap is finite under stationarity, namely
    \begin{equation}\label{eq:332n}
\EE_{\pi_n} [Y_{i-1}(t)]^{k+1} <\infty \mbox{ for all } t \ge 0.
    \end{equation}
    Consider now the $i^{th}$ gap. 
Let $\psi_i, \psi_i^{(m)}:\RR_+^{n-1}\to\RR$  as $\psi_i(\ybf)=y_i^k$, $\psi_i^{(m)}(\ybf)=y_i^k\wedge m$.
Note that 
$$\cll_n\psi_i(\ybf)= A(\ybf) + B(\ybf),$$ where $$A(\ybf):= y_i\int_0^1[u^ky_i^k - y_i^k]\; du,\; B(\ybf):= y_{i-1}\int_0^1 [(y_i+uy_{i-1})^k-y_i^k]\; du.$$ 
Then
$$
A(\ybf) = -\frac{k}{k+1} y_i^{k+1}
$$
and
$$
B(\ybf) \le y_{i-1}\int_0^1 (y_i+uy_{i-1})^k\; du
\le \frac{1}{k+1} (y_i+y_{i-1})^{k+1}.
$$
Using a weighted Young's inequality it now follows that there are $a_k, h_k>0$ such that
\begin{equation}\label{eq:359n}
\cll_n\psi_i(\ybf)= A(\ybf) + B(\ybf) \le -a_k y_{i}^{k+1} + h_k y_{i-1}^{k+1}, \mbox{ for all } \ybf \in \RR_+^{n-1}.
\end{equation}
Now, for $m \in \NN$, define
$G^{i}_t$, $G^{i, (m)}_t$ as in \eqref{eq:ggm} replacing there $\psi_1, \psi_1^{(m)}$ with
$\psi_i, \psi_i^{(m)}$, respectively.
Then, by a localization argument, as in \eqref{eq:1228n} we get
\begin{equation}\label{eq:1231n}
\EE_{\ybf}[\psi_i(\bfY(t))]
- \psi_1(\bfy) \le h_k \int_0^{t} \EE_{\ybf}(Y_{i-1}^{k+1}(s)) ds.
\end{equation}
This, similar to before, shows that
$$
G^{i, (m)}_t(\ybf) + \frac{h_k}{t} \int_0^{t} \EE_{\ybf}(Y_{i-1}^{k+1}(s)) ds \ge 0.
$$
Using the fact that $\EE_{\pi_n}(Y_{i-1}^{k+1}(t))<\infty$ we now get by Fatou's lemma and stationarity of $\pi_n$  that
\begin{equation}\int G^{i}_t(\ybf) \pi_n(d\ybf) \le
\liminf_{m\to \infty} \int G_t^{i,(m)}(\ybf) \pi_n(dy) =0. \label{eq:331n}
\end{equation}
Using \eqref{eq:359n}, a localization argument and monotone convergence, as below \eqref{eq:1030n}, now shows that
\begin{equation*}
\frac{1}{t}\EE_{\ybf} \int_0^{t}
 [Y_i(s)]^{k+1} ds \le
 \frac{1}{a_k}G^{i}_t(y) + \frac{h_k}{a_k}
 \frac{1}{t}\EE_{\ybf} \int_0^{t}
 [Y_{i-1}(s)]^{k+1} ds.
\end{equation*}
Integrating with respect to $\pi_n$ in the above inequality, and using \eqref{eq:331n} and \eqref{eq:332n} gives
\begin{equation*}
\EE_{\pi_n}[Y_i(t)]^{k+1}  \le
 \frac{h_k}{a_k}
 \EE_{\pi_n}[Y_{i-1}(t)]^{k+1} <\infty.
\end{equation*}
    The result follows.
\hfill
\qedsymbol

\paragraph{Acknowledgements.}
SB and DI were supported in part by the NSF-CAREER award (DMS-2141621).
AB was supported in part by the NSF (DMS-2134107, DMS-2152577). SB and AB were supported by the RTG award (DMS-2134107) from the NSF.
Part of this work was done while AB  was in residence at the Simons Laufer Mathematical Sciences Institute in
Berkeley, California, during Fall 2025.
This visit was supported by the National Science Foundation under Grant No. DMS-1928930.

{\footnotesize 
\bibliographystyle{is-abbrv}
\bibliography{references}
}

{\footnotesize 
%
%
%
}

\vspace{\baselineskip}
\vspace{\baselineskip}

\noindent{\scriptsize {\textsc{\noindent S. Banerjee, A. Budhiraja, and D. Imon,\newline
Department of Statistics and Operations Research\newline
University of North Carolina\newline
Chapel Hill, NC 27599, USA\newline
email: sayan@email.unc.edu
\newline
email: budhiraj@email.unc.edu
\newline
email: idilshad@unc.edu
\vspace{\baselineskip} } }}

\end{document}